\documentclass{article}
\usepackage{}

\usepackage{amssymb}
\usepackage{latexsym}
\usepackage{extarrows}
\usepackage[dvips]{graphicx}
\usepackage{graphicx}

\newtheorem{theo}{Theorem}[section]

\newtheorem{remark}[theo]{Remark}

\newtheorem{lemma}[theo]{Lemma}

\newtheorem{coro}[theo]{Corollary}
\newtheorem{con}[theo]{Conjecture}
\newtheorem{prop}[theo]{Proposition}
\newtheorem{fact}[theo]{Fact}
\newtheorem{defi}[theo]{Definition}
\newtheorem{algorithm}[theo]{Algorithem}

\newcommand{\qed}{\hspace*{\fill} \rule{7pt}{7pt}}

\begin{document}
\date{}
\title{ On a conjecture of Hefetz and Keevash on  Lagrangians of intersecting hypergraphs and Tur\'an numbers}
\author{ Biao Wu \thanks{ College of Mathematics and Econometrics, Hunan University, Changsha, 410082, P.R. China. Email: wubiao@hnu.edu.cn.}  \and{ Yuejian Peng \thanks{Corresponding author. Institute of Mathematics, Hunan University, Changsha, 410082, P.R. China. Email: ypeng1@hnu.edu.cn \ Supported in part by National Natural Science Foundation of China (No. 11671124).}} \and Pingge Chen \thanks{ College of Mathematics and Econometrics, Hunan University, Changsha, 410082, P.R. China. Email: chenpingge@hnu.edu.cn. }
}

\maketitle
\begin{abstract}
Let $S^r(n)$ be the $r$-graph on $n$ vertices with parts $A$ and $B$, where the edges consist of all $r$-tuples with $1$
 vertex in $A$ and $r-1$ vertices in $B$, and the sizes of $A$ and $B$ are chosen to maximise the number of edges.
 Let $M_t^r$ be the $r$-graph with $t$ pairwise disjoint edges.
 Given an $r$-graph $F$ and a positive integer $p\geq |V(F)|$, we define the {\em extension} of $F$, denoted by $H_{p}^{F}$ as follows:
 Label the vertices of $F$ as $v_1,\dots,v_{|V(F)|}$. Add new vertices $v_{|V(F)|+1},\dots,v_{p}$.
  For each pair of vertices $v_i,v_j, 1\le i<j \le p$ not contained in an edge of $F$,
 we add a set $B_{ij}$ of $r-2$ new vertices and the edge $\{v_i,v_j\} \cup B_{ij}$,
 where the $B_{ij}$ 's are pairwise disjoint over all such pairs $\{i,j\}$. Hefetz and Keevash conjectured that the Tur\'an number of the extension of $M_2^r$ is
 ${1 \over r}n\cdot{{r-1\over r}n \choose r-1}$ for $r \ge 4$ and sufficiently large $n$.
 Moreover, if $n$ is sufficiently large and $G$ is an $H_{2r}^{M_2^r}$-free $r$-graph with $n$ vertices
 and ${1 \over r}n\cdot{{r-1\over r}n \choose r-1}$ edges, then $G$ is isomorphic to $S^r(n)$.
 In this paper, we confirm the above conjecture for $r=4$.
\end{abstract}

\section{Introduction}

For a set $V$ and a positive integer $r$ we denote by $V^{(r)}$ the family of all $r$-subsets of $V$. An {\em $r$-uniform graph} or {\em $r$-graph $G$} consists of a vertex set $V(G)$ and an edge set $E(G) \subseteq V(G) ^{(r)}$. We sometimes write the edge set of $G$ as $G$. Let $|G|$ denote the number of edges of $G$. An edge $e=\{a_1, a_2, \ldots, a_r\}$ will be simply denoted by $a_1a_2 \ldots a_r$. An $r$-graph $H$ is  a {\it subgraph} of an $r$-graph $G$, denoted by $H\subseteq G$ if $V(H)\subseteq V(G)$ and $E(H)\subseteq E(G)$.
A subgraph $G$ {\em induced} by $V'\subseteq V$, denoted as $G[V']$, is the $r$-graph with vertex set $V'$ and edge set $E'=\{e\in E(G):e \subseteq V'\}$.
Let $K^{r}_t$ denote the {\em complete $r$-graph} on $t$ vertices. 

Given an $r$-uniform hypergraph $F$,  an $r$-uniform hypergraph $G$ is called $F$-free if it does not contain a copy of $F$ as a subgraph.  The {\em Tur\'an number} of $F$, denoted by $ex(n,F)$, is the maximum number of edges in an $F$-free $r$-uniform hypergraph on $n$ vertices.
 An averaging argument of Katona, Nemetz and Simonovits \cite{KNS} implies the sequence $ ex(n,F) / {n \choose r } $ decreases. So $\lim_{n\rightarrow\infty}  ex(n,F) / {n \choose r }$ exists. The {\em Tur\'{a}n density} of $F$ is defined as $$\pi(F)=\lim_{n\rightarrow\infty} { ex(n,F) \over {n \choose r } }.$$
 For 2-graphs, Erd\H{o}s-Stone-Simonovits determined the Tur\'an densities of all graphs except bipartite graphs. Very few results are known for hypergraphs and a survey on this topic can be found in Keevash's survey paper \cite{Keevash}. Lagrangian has been a useful tool in estimating the Tur{\'a}n density of a hypergraph.

\begin{defi}
For  an $r$-graph $G$ with the vertex set $[n]$,
edge set $E(G)$ and a weighting $\vec{x}=(x_1,\ldots,x_n) \in \mathbb R^n$,
define the Lagrangian function of $G$ as
$$\lambda (G,\vec{x})=\sum_{e \in E(G)}\prod\limits_{i\in e}x_{i}.$$
\end{defi}
The {\em Lagrangian} of
$G$, denoted by $\lambda (G)$, is defined as
 $$\lambda (G) = \max \{\lambda (G, \vec{x}): \vec{x} \in \Delta \},$$
where $$\Delta=\{\vec{x}=(x_1,x_2,\ldots ,x_n) \in \mathbb{R}^{n}: \sum_{i=1}^{n} x_i =1, x_i \ge 0 {\rm \ for \ every\ } i\in[n] \}.$$

The value $x_i$ is called the {\em weight} of the vertex $i$ and a weighting $\vec{x} \in {\Delta}$ is called a {\em feasible weighting}.
A weighting $\vec{y}\in {\Delta}$ is called an {\em optimum weighting} for $G$ if $\lambda (G, \vec{y})=\lambda(G)$.

\begin{fact}
If $G' \subseteq G$ then $\lambda(G') \le \lambda(G)$.
\end{fact}

Given an $r$-graph $F$, the {\em Lagrangian density } $\pi_{\lambda}(F)$ of $F$ is defined as
$$\pi_{\lambda}(F)=\sup \{r! \lambda(G): G \;is\; an\; F\text{-}free \;r\text{-}graph\}.$$
Usually, the difficulty part of obtaining Tur\'an density is to get a good upper bound. The following remark says that the Tur\'an density of an $r$-graph is no more than its Lagrangian density.
\begin{remark} {\em (see Remark 1.2 in \cite{JPW})}
$\pi(F)\le \pi_{\lambda}(F).$
\end{remark}

The Lagrangian method for hypergraph Tur\'an problems were developed independently by Sidorenko \cite{Sidorenko-87} and Frankl-F\"uredi \cite{FF}, generalizing work of Motzkin-Straus \cite{MS} and Zykov \cite{Z}. More recent developments of the method were obtained by Pikhurko \cite{p} and Norin and Yepremyan \cite{NY}. More recent results based on these developments will be introduced later. 

Let $r\ge 3$, $F$ be an $r$-graph and $p\ge |V(F)|$. Let ${\mathcal{K}}_p^F $ denote the family of $r$-graphs $H$ that contains a set $C$ of $p$ vertices, called the {\em core}, such that the subgraph of $H$ induced by $C$ contains a copy of $F$ and such that every pair of vertices in $C$ is {\em covered} in $H$ (A pair of vertices $i,j$ of $H$ is {\em covered} if there exists an edge of $H$ containing both $i,j$.).
Let $H_p^F $ be a member of ${\mathcal{K}}_p^F $obtained as follows. Label the vertices of $F$ as $v_1,\dots,v_{|V(F)|}$. Add new vertices $v_{|V(F)|+1},\dots,v_{p}$. Let $C=\{v_{1},\dots,v_{p}\}$. For each pair of vertices $v_i,v_j \in C$ not covered in $F$, we add a set $B_{ij}$ of $r-2$ new vertices and the edge $\{v_i,v_j\} \cup B_{ij}$, where the $B_{ij}$'s are pairwise disjoint over all such pairs $\{i,j\}$. We call $H_p^F $ the {\em extension} of $F$.

Frankl and F\"uredi \cite{FF} conjectured that for all $r\geq 4$, if $n\geq n_0(r)$ is sufficiently large then $ex(n,\mathcal {K}_{r+1}^{L})=ex(n,H_{r+1}^{L})$, where $L$ is the graph on $r+1$ vertices consisting of two edges sharing $r-1$ vertices. Let $T_{r}(n,l)$ denote the balanced complete $l$-partite $r$-graph on $n$ vertices. Pikhurko \cite{Pikhurko} proved the conjecture for $r=4$, showing that $ex(n,\mathcal {K}_{5}^{L})=ex(n,H_{5}^{L})=e(T_{4}(n,4))$, with the $T_{4}(n,4)$ being the unique extremal graph. Recently, Norin and Yepremyan \cite{NY} proved the conjecture for $r=5$ and $r=6$, moreover, extremal graphs are blowups of the unique $(11,5,4)$ and $(12,6,5)$ Steiner systems for $r=5$ and $r=6$, respectively. For all $n,p,r$, Mubayi \cite{M} and Pikhurko \cite{Pikhurko} showed that $ex(n,\mathcal {K}_{p}^{F})=ex(n,H_{p}^{F})=e(T_{r}(n,p-1))$ with the unique extremal graph being $T_{r}(n,p-1)$, where $F$ is the $r$-uniform empty graph. Mubayi and Pikhurko \cite{MP} showed that for all $r\geq 3$ and all sufficiently large $n$, $ex(n,H_{r+1}^{f})=e(T_{r}(n,r))$, where $f$ is a single $r$-set. Moreover, $T_{r}(n,r)$ is the unique extremal graph. Brandt-Irwin-Jiang \cite{BIJ} and independently Norin and Yepremyan \cite{NY2} showed that for a large family of $r$-graphs $F$ and sufficiently large $n$, $ex(n,H_{p}^{F})=e(T_{r}(n,p-1))$ with the unique extremal graph being $T_{r}(n,p-1)$.

Let $M_t^r$ be the $r$-graph with $t$ pairwise disjoint edges, called {\em $r$-uniform $t$-matching}. Hefetz and Keevash in \cite{HK} determined  the Lagrangian density of $M_2^3$, and showed $ex(n,H_{6}^{M_{2}^{3}})=e(T_{3}(n,5))$ for large $n$ and $T_{3}(n,5)$ is the unique extremal graph. More generally, Jiang-Peng-Wu in \cite{JPW} determined the Lagrangian density of $M_t^3$ and showed that $ex(n,H_{3t}^{M_{t}^{3}})=e(T_{3}(n,3t-1))$ for large $n$ and $T_{3}(n,3t-1)$ is the unique extremal graph.

Let $S^r(n)$ be the $r$-graph on $[n]$ with parts $A$ and $B$, where the edges consist of all $r$-tuples with $1$ vertex in $A$ and $r-1$ vertices in $B$, and the sizes of $A$ and $B$ are chosen to maximise the number of edges (so $|A|\approx n/r)$. Write $s^r(n)=|S^r(n)|$.
When $r=4$, then $|A|=\lfloor {n \over 4} \rfloor$ for $n \equiv 0$ or $1$ or $2$ {\rm (mod $4$)} and $|A|=\lfloor {n \over 4} \rfloor$ or $\lceil{n \over 4}\rceil$ for $n \equiv 3$ {\rm (mod $4$)}. In \cite{HK}, Hefetz and Keevash proposed the following conjecture.
\begin{con} {\em (D. Hefetz and P. Keevash, \cite{HK})}\label{HFcon}
$ex(n,H^{M_2^r}_{2r})={1 \over r}n\cdot{{r-1 \over r}n \choose r-1}$ for $r \ge 4$ and sufficiently large $n$. Moreover, if $n$ is sufficiently large and $G$ is an $H_{2r}^{M_{2}^{r}}$-free $r$-graph on $[n]$ with ${1 \over r}n\cdot{{r-1 \over r}n \choose r-1}$ edges, then $G \cong S^r(n)$. 
\end{con}

In this paper, we confirm the above conjecture for $r=4$.

\begin{theo}\label{theorem}
 $ex(n,H^{M_2^4}_{8})=\lfloor{1 \over 4}n\rfloor\cdot{\lceil{3 \over 4}n\rceil \choose 3}$ for sufficiently large $n$. Moreover, if $n$ is sufficiently large and $G$ is an $H_{8}^{M_{2}^{4}}$-free $4$-graph on $[n]$ with $\lfloor{1 \over 4}n\rfloor\cdot{\lceil{3 \over 4}n\rceil \choose 3}$ edges, then $G\cong S^4(n)$.
\end{theo}

Let $\mathcal{S}_n$ be the $4$-graph on vertex set $[n]$ with all edges containing a fixed vertex, we call it a {\em star}. $\mathcal{S}_n$ is denoted by $\mathcal{S}$ while $n$ goes to infinity. We apply the Lagrangian method  in the proof and show the following result that the maximum  Lagrangian among all $M_2^{4}$-free $4$-graphs is uniquely achieved by $\mathcal{S}$.
\begin{theo}\label{theoremM^4_2}
Let $\mathcal{F}$ be an $M_2^4$-free $4$-graph on $[n]$. If $\mathcal{F}\nsubseteq \mathcal{S}_n$ then $\lambda(\mathcal{F}) <0.0169< \frac{9}{512}$.
Otherwise, $\lambda(\mathcal{F}) \le {9(n-2)(n-3) \over 512(n-1)^2}$.
\end{theo}

\begin{coro}
$\pi_{\lambda}(M_2^4) = 4!\lambda(\mathcal{S})={27 \over 64}$.
\end{coro}

\section{Preliminaries}
Given an $r$-graph $G$ and a set $T$ of vertices, the {\em link} of $T$ in $G$, denoted by $L_G(T)$, is the hypergraph with edge set $\{e\in {V(G) \choose r-|T|}:e \cup T \in E(G)\}$.
If $T=\{i\}$, we write $L_G(\{i\})$ as $L_G(i)$ for short.
Let $i, j \in V(G)$, denote $$L_G(i \setminus j)=\{e \in {V(G) \choose r-1} : j \notin e, e \cup \{i\} \in E(G) {\rm \ and \ } e \cup \{j\} \notin E(G)\}.$$
We sometimes drop the subscript $G$.
If $L(i\setminus j)=L(j\setminus i)$ and  $\{i,j\}$ is not contained in any edge of $G$, then we say that $i$ and $j$ are {\em equivalent} and write $i\sim j$.
We say $G$ on vertex set $[n]$ is {\em left-compressed} if for every $i,j$, $1\le i < j \le n$, $L_G(j \setminus i)= \emptyset$.
 Given $i,j \in V(G)$, define
$$\pi_{ij}(G)=\left(E(G)\setminus \{\{j\}\cup F: F\in  L_G(j \setminus i)\} \right) \bigcup \{ \{i\} \cup F: F \in L_G(j \setminus i) \}.$$
By the definition of $\pi_{ij}(G)$, it's easy to see the following fact.

\begin{fact}\label{compression-preserve}
Let $G$ be an $r$-graph on vertex set $[n]$. Let $\vec{x}=(x_1,x_2,\dots,x_n)$ be a feasible weighing of $G$. If $x_i \ge x_j$, then $\lambda(\pi_{ij}(G),\vec{x})\ge \lambda(G,\vec{x})$.
\end{fact}
The following lemma plays an important role in the proof.

\begin{lemma}{\rm (see e.g. \cite{frankl-survey})}\label{t-free}
Let $G$ be an $M_t^r$-free $r$-graph on vertex set $[n]$. Then for every pair $i,j$ with $1\le i \neq j \le n$, $\pi_{ij}(G)$ is $M_t^r$-free.
\end{lemma}

An $r$-graph $G$ is {\em dense} if and only if every proper subgraph $G'$ of $G$ satisfies $\lambda (G') < \lambda (G)$. This is equivalent to that all optimum weightings of $G$ are in the interior of ${\Delta}$, which means no coordinate in an optimum weighting is zero.
\begin{algorithm}\label{left-compression}{\rm (Dense and compressed \cite{JPW})}

\noindent{\bf Input:}   An $M_t^r$-free $r$-graph $G$ on $[n]$.

\noindent{\bf Output:}  A dense and left-compressed $M_t^r$-free $r$-graph with non-decreasing Lagrangian and an optimum weighting  $\vec{x}$ satisfying $x_{i}\geq x_{j}$ if $i\leq j$.

\noindent{\bf Step 1.} If $G$ is not dense, then replace $G$ by a dense subgraph with
the same Lagrangian. Otherwise, go to Step 2.

\noindent{\bf Step 2.} Let  $\vec{x}$ be an optimum weighting of $G$. If $G$ is left-compressed,
then terminate. Otherwise,
relabel the vertices of $G$ such that $x_i\geq x_j$ for all $i < j$ if necessary.
So there exist vertices  $i,j$ such that $x_i\geq x_j$ and $L_G(j \setminus i)\neq \emptyset $,
then replace $G$ by $\pi_{ij}(G)$ and go to step 1.
\end{algorithm}
Note that the algorithm terminates after finite many steps since Step 1 reduces the number of vertices by at least 1  each time and Step 2 reduces the parameter $s(G)=\sum_{e\in G}\sum_{i\in e}i$ by at least 1  each time.
Applying Fact \ref{compression-preserve}, Lemma \ref{t-free}
and the fact that Algorithms \ref{left-compression} terminate after finite many steps,  we get
the following lemma.

\begin{lemma} \label{compression-lemma}
Let $G$ be an $M^r_t$-free $r$-graph. Then there exists an $M^r_t$-free dense and left-compressed $r$-graph $G'$ with $\vert V(G')\vert \le \vert V(G)\vert$
such that $\lambda(G')\geq \lambda(G)$. \qed
\end{lemma}
We need the following result to estimate the Lagrangians of some hypergraphs.

\begin{theo}{\rm (\cite{PZ})} \label{PZ} Let $m$ and $l$ be positive integers satisfying ${l-1 \choose 3} \le m \le {l-1 \choose 3} + {l-2 \choose 2}$. Let $G$ be a $3$-graph with $m$ edges and $G$ contains a clique of order  $l-1$. Then $\lambda(G) = \lambda([l-1]^{(3)})$.
\end{theo}

Let $G$ be an $r$-graph on $[n]$ and $\vec{x}=(x_1,x_2,\dots,x_n)$ be a weighing of $G$. Fix $i \in [n]$, denote
$$ L_{G}(x_i)=\frac{\partial \lambda (G, \vec{x})}{\partial x_i}=\sum_{i \in e \in E(G)}\prod\limits_{j\in e\setminus \{i\}}x_{j}.$$

\begin{fact} {\em (\cite{FR})}\label{fact2}
Let $G$ be an $r$-graph on $[n]$. Let $\vec{x}=(x_1,x_2,\dots,x_n)$ be an optimum weighing for $G$ with $k\leq n$ nonzero weights $x_1,x_2,\cdots,x_k$. Then, for every $\{i,j\} \in [k]^{(2)}$,
\newline
{\em (1)} $L_{G}(x_i)=r\lambda(G)$ ;
\newline
{\em (2)} The pair $i$ and $j$ is covered.
\end{fact}

\begin{fact} {\em (\cite{FR})}\label{fact1}
Let $G=(V,E)$ be a dense $r$-graph. Then every pair  of vertices $i,j\in V$ is covered.
\end{fact}

\begin{lemma}{\em (\cite{HK})}\label{lemma-CP}
Let $G$ be an $r$-graph on vertex set $[n]$. If the pair $\{i,j\} \subseteq [n]$ is not covered, then $\lambda(G)=\max\{\lambda(G\setminus\{i\}), \lambda(G\setminus\{j\})\}$. Further more, if $L_G(i) \subseteq L_G(j)$ then $\lambda(G)=\lambda(G\setminus\{i\})$.
\end{lemma}

Given disjoint sets of vertices $V_1,\ldots, V_s$, denote $\Pi_{i=1}^s V_i=
V_1\times V_2\times\ldots \times V_s=\{(x_1,x_2,\ldots, x_s): \forall i=1,\ldots,s, x_i\in V_i\}$.
We will also use $\Pi_{i=1}^s V_i$ to  denote the set of the corresponding unordered $s$-sets.
Let $F$ be a hypergraph on $[m]$, a {\it blowup} of $F$ is a hypergraph $G$
whose vertex set can be partitioned into $V_1,\ldots, V_m$ such that
$E(G)=\bigcup_{e\in F} \prod_{i\in e} V_i$.

\begin{coro}\label{coroblowup}{\em(\cite{FR})}
Given an $r$-graph $F$. Let $G$ be a blowup of $F$, then $\lambda(G)=\lambda(F)$.
\end{coro}

\begin{lemma}{\em (\cite{HK})}\label{Equivalent}
Let $G$ be an $r$-graph on vertex set $[n]$. If $L(i\setminus j)=L(j\setminus i)$, then there is an optimum weighting $\vec{x}=\{x_1,x_2,\cdots, x_n\}$ such that $x_i=x_j$.
\end{lemma}

\begin{lemma} \label{K_5^3-2}
Let $G$ be a $3$-graph obtained by removing two edges intersecting at two vertices from $K_5^3$. Then $\lambda(G)\le 0.0673<{9 \over 128}$.
\end{lemma}
\begin{proof}
Without loss of generality, suppose that $G=K_5^3\setminus \{245,345\}$. Let $\vec{x}=(x_1,\dots,x_5)$ be an optimum weighting of $G$. By Lemma \ref{Equivalent}, we can assume that $x_1=a$, $x_2=x_3={b \over 2}$ and $x_4=x_5={c \over 2}$. So $a+b+c=1$. Then
\begin{eqnarray*}
 \lambda(G,\vec{x})&=& {ab^2 \over 4}+abc+{ac^2 \over 4}+{b^2c \over 4} \\
 &=& {a \over 4}(b+c)^2+{1 \over 2}abc + {b^2c \over 4} \\
 &=& {a \over 4}(1-a)^2 + {bc \over 4}(2a+b) \\
 &=& {a \over 4}(1-a)^2 + {2.2b \times3.2c \times (2a+b)  \over 4\times 2.2\times 3.2} \\
 &\le& {a \over 4}(1-a)^2 + {1  \over 4\times 2.2\times 3.2}\left(\frac{2a+3.2b+3.2c}{3}\right)^3 \\
 &=& {a \over 4}(1-a)^2 + {1  \over 4\times 2.2\times 3.2\times 27}\left(3.2-1.2a\right)^3 \\
 &=& {2943a^3-5724a^2+2394a+512 \over 4 \times 27 \times 110}
 .
\end{eqnarray*}
Let $g(a)=2943a^3-5724a^2+2394a+512$. Then $g'(a)=0$ implies $a=\frac{212-5\sqrt{638}}{327}$. So $ \lambda(G,\vec{x})\leq{g(\frac{212-5\sqrt{638}}{327}) \over 4 \times 27 \times 110}< 0.0673.$
\qed
\end{proof}
\medskip

\section{Lagrangians of intersecting 4-graphs}
We first calculate $\lambda(\mathcal {S})$.
\begin{lemma}\label{F_0}
$\lambda(\mathcal {S}_n)= {9(n-2)(n-3) \over 512(n-1)^2}$ for $n\ge 4$ and $\lambda(\mathcal{S})= {9 \over 512}$.
\end{lemma}
\begin{proof}
Note that $\mathcal {S}_n=\{1ijk: 2\leq i<j<k\leq n\}$. Let $\vec{x}=(x_1,\dots,x_n)$ be a feasible weighting of $\mathcal{S}_n$, then
$$\lambda(\mathcal {S}_n,\vec{x})=x_1\sum_{2\le i < j < k \le n} x_ix_jx_k\leq x_1{n-1 \choose 3}\left(\frac{1-x_1}{n-1}\right)^3\leq {9(n-2)(n-3) \over 512(n-1)^2},$$
equality holds if and only if $ x_1={1 \over 4}$ and $x_2= \dots = x_n ={3 \over 4(n-1)} $.
So $\lambda(\mathcal{S})=\lim_{n\rightarrow +\infty}\mathcal \lambda(\mathcal {S}_{n},\vec{x})= {9 \over 512}.$
\qed
\end{proof}

\medskip
\noindent{\bf \em Proof of Theorem \ref{theoremM^4_2}}.
Let $\mathcal {F}$ be an $M_{2}^{4}$-free 4-graph on $[n]$. By Lemma \ref{compression-lemma}, we may assume that $\mathcal {F}$ is left-compressed and dense.
If $n \le 7$, then $\mathcal {F} \subseteq K_7^4$. Therefore $\lambda(\mathcal{F})\le \lambda(K_7^4) = {5 \over 343}<0.0169$.
Now assume that $n \ge 8$.
Let $\vec{x}=(x_1,\dots,x_n)$ be an optimum weighting of $\mathcal{F}$, satisfying $x_1\ge x_2 \ge \dots\ge x_n>0$.
The proof classifies such $4$-graphs into several cases and verifies the required bound in each case.

Since $\mathcal{F}$ is left-compressed and dense, then every pair $i,j$ with $3 \le i < j \le n$ satisfies $12ij \in \mathcal{F}$. We claim that $\{1,2\}$ is a vertex cover of $\mathcal{F}$ (i.e., every edge of $\mathcal{F}$ contains 1 or 2), otherwise $3456 \in \mathcal{F}$, then $\{3456 ,1278\}$ forms a copy of $M_2^4$ in $\mathcal{F}$, a contradiction. Furthermore, we have $2468 \notin  \mathcal{F}$, otherwise $\{2468 ,1357\}$ forms a copy of $M_2^4$ in $\mathcal{F}$, a contradiction.

\medskip
Case 1. $2567 \in \mathcal {F}$.

Since $\mathcal{F}$ is $M_{2}^{4}$-free, we have $1348 \notin \mathcal{F}$. Then
$$\mathcal{F} \subseteq \mathcal{F}_1=\{12ij: 3 \le i < j \le n\}\cup \{ijkl:i \in [2], 3 \le j <k < l \le 7\}.$$
\begin{lemma}\label{F_1}
 $\lambda(\mathcal{F}_1)\le {1 \over 108}<0.0169.$
\end{lemma}
\begin{proof}
Let $\vec{x}=(x_1,\dots,x_n)$ be an optimum weighting of $\mathcal{F}_1$. Note that $L_{\mathcal{F}_1}(8)=\{12i: 3 \le i \le n,i\neq8\}$. By Fact \ref{fact2} (1), we have $\lambda(\mathcal{F}_1)= {1 \over 4}\lambda(L_{\mathcal{F}_1}(8),\vec{x}) \le {1 \over 4}\lambda(L_{\mathcal{F}_1}(8))={1 \over 108}.$
\qed
\end{proof}
\medskip

Case 2. $2567 \notin \mathcal{F}$ and $2467 \in \mathcal{F}$.

In this case, we have $1358 \notin \mathcal{F}$. Then
$$\mathcal{F} \subseteq \mathcal{F}_{2}=\{12ij, 134k, 13lm, 14lm, 1567, 234k, 23lm, 24lm: 3 \le i < j \le n, 5 \le k \le n, 5 \le l < m \le 7\}.$$
\begin{lemma}\label{F_{21}}
$\lambda(\mathcal{F}_2)\le {1 \over 64}<0.0169$.
\end{lemma}
\begin{proof}
Let $\vec{x}=(x_1,\dots,x_n)$ be an optimum weighting of $\mathcal{F}_2$. Note that $L_{\mathcal{F}_2}(8)=\{12i,134,234: 3 \le i \le n,i\neq8\}$. Since the pair of vertices $\{4,5\}$ is not covered in $L_{\mathcal{F}_2}(8)$, we can assume that $L_{\mathcal{F}_2}(8)$ is on $[4]$ by Lemma \ref{lemma-CP}. Moreover, $L_{\mathcal{F}_2}(8)[[4]]=\{123,124,134,234\}$.
By Fact \ref{fact2} (1), we have $\lambda(\mathcal{F}_2)= {1 \over 4}\lambda(L_{\mathcal{F}_2}(8),\vec{x}) \le {1 \over 4} \lambda(L_{\mathcal{F}_2}(8))={1 \over 4} \lambda(K_4^3)={1 \over 64}.$
\qed
\end{proof}
\medskip

Case 3. $2467 \notin \mathcal{F}$ and $2368 \in \mathcal{F}$.

In this case, we have $1457 \notin \mathcal{F}$. Then
\[
\begin{split}
\mathcal{F} \subseteq \mathcal{F}_{3}= \{12i_1j_1, 13i_2j_2,1456, 23i_3j_3, 2456: 3 \le i_1 < j_1\le n,4 \le i_2 < j_2\le n, 4 \le i_3 < j_3\le n,\},
\end{split}
\]

\begin{lemma}\label{F_{22}}
$\lambda(\mathcal{F}_3)\le {1 \over 64}<0.0169$.
\end{lemma}
\begin{proof}
Let $\vec{x}=(x_1,\dots,x_n)$ be an optimum weighting of $\mathcal{F}_3$.  Note that $$L_{\mathcal{F}_3}(8)=\{12i,13j,23k: 3\le i \le n,4 \le j,k \le n,i,j,k\neq8\}.$$
Since the pair of vertices $\{4,5\}$ is not covered in $L_{\mathcal{F}_3}(8)$. By Fact \ref{fact2} (1) and Lemma \ref{lemma-CP}, we have $\lambda(\mathcal{F}_3)= {1 \over 4}\lambda(L_{\mathcal{F}_3}(8),\vec{x}) \le {1 \over 4} \lambda(L_{\mathcal{F}_3}(8))={1 \over 4} \lambda(K_4^3)={1 \over 64}.$
\qed
\end{proof}
\medskip

Case 4. $2467, 2368 \notin \mathcal{F}$ and $2458 \in \mathcal{F}$.

In this case, we have $1367 \notin \mathcal{F}$. Then
$$\mathcal{F} \subseteq \mathcal{F}_{4}=\{12ij,134k,135l,145l,234k,235l,245l:3 \le i<j \le n, 5 \le k \le n, 6 \le l \le n\}.$$
\begin{lemma} \label{F_{23}}
$\lambda(\mathcal{F}_4)\le {4\over 243}<0.0169$.
\end{lemma}
\begin{proof}
Let $\vec{x}=(x_1,\dots,x_n)$ be an optimum weighting of $\mathcal{F}_4$. By Lemma \ref{Equivalent}, we can assume that $x_1=x_2=a$, $x_3=x_4=x_5=b$ and $x_6+ \dots + x_n=c$. So $2a+3b+c=1$. Note that $L_{\mathcal{F}_4}(1)=\{ 2ij,34k,35l,45l:3\le i < j \le n, 5 \le k \le n, 6 \le l \le n \}$. Then
\begin{eqnarray*}
 \lambda(L_{\mathcal{F}_4}(1),\vec{x})&\leq& a{(1-2a)^2 \over 2}+b^3+3b^2c \\
 &=& {a(1-2a)^2 \over 2}+b^2(b+3c)\\
  &=& {a(1-2a)^2 \over 2}+\frac{1}{16}\times4b\times 4b\times(b+3c)\\
 &\le& {a(1-2a)^2 \over 2}+\frac{1}{16}(3b+c)^3.
\end{eqnarray*}
Since $3b+c=1-2a$, we have
\begin{eqnarray*}
 \lambda(L_{\mathcal{F}_4}(1),\vec{x})
 &\le& {a(1-2a)^2 \over 2}+{(1-2a)^3 \over 16}\\
 &=& {(1-2a)^2 \over 16}(1+6a)  \\
 &=& \frac{3}{2}\times {(1-2a)^2 \over 16}(\frac{2}{3}+4a) \\
 &\le&\frac{3}{2}\times\frac{1}{16}\times(\frac{8}{9})^3\\
 &=& {16 \over 243}.
\end{eqnarray*}
Hence $ \lambda(\mathcal{F}_4)={1 \over 4}\lambda(L_{\mathcal{F}_4}(1),\vec{x})\le {4 \over 243}$.
\qed
\end{proof}
\medskip

Case 5. $2467, 2368, 2458 \notin \mathcal{F}$, $2367 \in \mathcal{F}$.

In this case, we have $1458 \notin \mathcal{F}$. Then
\[
\begin{split}
\mathcal{F} \subseteq \mathcal{F}_{5}= \{12ij, 13kl,1456, 1457, 1567, 2345,  234m, 235m, 2367, 2456, 2457:\\ 3 \le i < j \le n, 4\le k < l \le n, 6 \le m \le n\}.
\end{split}
\]
\begin{lemma}\label{F_{31}}
$\lambda(\mathcal{F}_5)\le {1 \over 64}<0.0169$.
\end{lemma}
\begin{proof}
Let $\vec{x}=(x_1,\dots,x_n)$ be an optimum weighting of $\mathcal{F}_5$. Note that $$L_{\mathcal{F}_5}(8)=\{12i,13j,234,235:  3 \le i \le n,4 \le j \le n,i,j\neq8\}.$$
By Fact \ref{fact2} (1) and Lemma \ref{lemma-CP}, we have $\lambda(\mathcal{F}_5)= {1 \over 4}\lambda(L_{\mathcal{F}_5}(8),\vec{x}) \le {1 \over 4} \lambda(L_{\mathcal{F}_5}(8))={1 \over 4} \lambda(K_4^3)={1 \over 64}.$
\qed
\end{proof}
\medskip

Case 6. $2458, 2367 \notin \mathcal{F}$ and $2457 \in \mathcal{F}$.

In this case, we have $1368 \notin \mathcal{F}$. Then
\[
\begin{split}
\mathcal {F}\subseteq\mathcal{F}_{6}= \{12ij, 134k, 135l,1367, 145l,1467, 1567, 234k, 235l,  2456, 2457:\\ 3 \le i < j \le n, 5 \le k \le n, 6 \le l \le n\}.
\end{split}
\]
\begin{lemma}\label{F_{32}}
$\lambda(\mathcal{F}_6)< 0.0169$.
\end{lemma}
\begin{proof}
Let $\vec{x}=(x_1,\dots,x_n)$ be an optimum weighting of $\mathcal{F}_6$. Note that $$L_{\mathcal{F}_6}(8)=\{12i,134,135,145,234,235: 3 \le i \le n, i\neq8\}.$$
Since the pair of vertices $\{5,6\}$ is not covered in $L_{\mathcal{F}_6}(8)$, we have
$$\lambda(\mathcal{F}_6)= {1 \over 4}\lambda(L_{\mathcal{F}_6}(8),\vec{x}) \leq{1 \over 4}\lambda(L_{\mathcal{F}_6}(8))=\frac{1}{4}\lambda(K_{5}^{(3)}\backslash \{245,345\})< 0.0169, $$
according to Fact \ref{fact2} (1), Lemma \ref{lemma-CP} and Lemma \ref{K_5^3-2}.
\qed
\end{proof}
\medskip

Case 7. $2457, 2367 \notin \mathcal{F}$ and $2358 \in \mathcal{F}$.

 In this case, we have $1467 \notin \mathcal{F}$. Then
$$\mathcal{F} \subseteq \mathcal{F}_{7}=\{12ij, 13kl, 145m,2345, 234m, 235m, 2456 :3 \le i < j \le n, 4 \le k < l \le n, 6 \le m \le n\}.$$
\begin{lemma}\label{F_{33}}
$\lambda(\mathcal{F}_7)< 0.0169$.
\end{lemma}
\begin{proof}
Let $\vec{x}=(x_1,\dots,x_n)$ be an optimum weighting of $\mathcal{F}_7$. Note that $$L_{\mathcal{F}_7}(8)=\{12i,13j,145,234,235: 3 \le i \le n,4 \le j \le n,i,j\neq8 \}.$$
Since the pair of vertices $\{5,6\}$ is not covered in $L_{\mathcal{F}_7}(8)$, we have
$$\lambda(\mathcal{F}_7)= {1 \over 4}\lambda(L_{\mathcal{F}_7}(8),\vec{x}) \leq{1 \over 4}\lambda(L_{\mathcal{F}_7}(8))=\frac{1}{4}\lambda(K_{5}^{(3)}\backslash \{245,345\})< 0.0169, $$
according to Fact \ref{fact2} (1), Lemma \ref{lemma-CP} and Lemma \ref{K_5^3-2}.
\qed
\end{proof}
\medskip

Case 8. $2457, 2367, 2358 \notin \mathcal{F}$, and $2357 \in \mathcal{F}$.

In this case, we have $1468 \notin \mathcal{F}$. Then
\[
\begin{split}
 \mathcal{F} \subseteq \mathcal{F}_{8}= \{12ij, 13kl, 145m, 1467, 234m, 2345, 2356, 2357, 2456:\\ 3 \le i < j \le n,4 \le k< l\le n, 6 \le m \le n\}.
\end{split}
\]
\begin{lemma}\label{F_{41}}
$\lambda(\mathcal{F}_8)\le \frac{1}{64}<0.0169$.
\end{lemma}
\begin{proof}
Note that $L_{\mathcal{F}_8}(8)=\{12i,13j,145,234: 3 \le i \le n, 4\le j \le n,i,j\neq8 \}$. Since the pair of vertices $\{5,6\}$ is not covered by any edge in $L_{\mathcal{F}_8}(8)$, we can assume that $L_{\mathcal{F}_8}(8)$ is on $[5]$ by Lemma \ref{lemma-CP}.
Moreover, $L_{\mathcal{F}_8}(8)[[5]]=\{123,124,125,134,135,145,234\}$, hence we have $ \lambda(L_{\mathcal{F}_{8}}(8)) = \lambda(K_4^3)= {1 \over 16}$ by Theorem \ref{PZ}. So $\lambda(\mathcal{F}_8)\le {1 \over 4}\lambda(L_{\mathcal{F}_{8}}(8)) = {1 \over 64}$.
\qed
\end{proof}
\medskip

Case 9. $2357 \notin \mathcal{F}$ and $2348 \in \mathcal{F}$.

In this case, we have $1567 \notin \mathcal{F}$. Then
$$\mathcal{F} \subseteq \mathcal{F}_{9}= \{12ij, 13kl, 14st, 234p, 2356, 2456: 3 \le i < j \le n, 4 \le k < l \le n,  5 \le s < t \le n, 5 \le p \le n\}.$$
\begin{lemma}\label{F_{42}}
$\lambda(\mathcal{F}_9)\le \frac{1}{64}<0.0169$.
\end{lemma}
\begin{proof}
Note that $L_{\mathcal{F}_9}(8)=\{12i,13j,14k,234: 3 \le i \le n,4 \le j \le n,5\le k \le n,i,j,k\neq8 \}$. Since the pair of vertices $\{5,6\}$ is not covered in $L_{\mathcal{F}_9}(8)$, we can assume that $L_{\mathcal{F}_9}(8)$ is on $[5]$ by Lemma \ref{lemma-CP}.
Moreover, $L_{\mathcal{F}_9}(8)[[5]]=\{123,124,125,134,135,145, 234\}$, hence we have $ \lambda(L_{\mathcal{F}_{9}}(8)) = \lambda(K_4^3)= {1 \over 16}$ by Theorem \ref{PZ}. So $\lambda(\mathcal{F}_9)\le {1 \over 4}\lambda(L_{\mathcal{F}_{9}}(8)) = {1 \over 64}$.
\qed
\end{proof}
\medskip

Case 10. $2357,2348 \notin \mathcal{F}$ and $2456 \in \mathcal{F}$.

In this case, we have $1378 \notin \mathcal{F}$. Then
\[
\begin{split}
\mathcal{F} \subseteq \mathcal{F}_{10}=\{12ij, 134k, 135l, 136m, 145l, 146m, 156m,2345, 2346, 2347, 2356,2456:\\
3 \le i < j \le n, 5\le k \le n, 6\le l\le n, 7\le m \le n\}.
\end{split}
\]
\begin{lemma}\label{F_{43}}
$\lambda(\mathcal{F}_{10})\le \frac{2}{135}<0.0169$.
\end{lemma}
\begin{proof}
Let $\vec{x}=(x_1,\dots,x_n)$ be an optimum weighting of $\mathcal{F}_{10}$. Note that $$L_{\mathcal{F}_{10}}(8)=\{12i, 134, 135, 136, 145, 146, 156: 3 \le i \le n,i\neq8\}.$$
Since the pair of vertices $\{6,7\}$ is not covered in $L_{\mathcal{F}_{10}}(8)$, then by Lemma \ref{lemma-CP}, $\lambda(L_{\mathcal{F}_{10}}(8))= \lambda(\{1ij:2\leq i<j\leq 6\})$.
Let $G=\{1ij:2\leq i<j\leq 6\}$ and $\vec{x}=(x_1,\dots,x_6)$ be an optimum weighting of $G$. By Lemma \ref{Equivalent}, we can assume that $x_1=x, x_2=\cdots=x_6=\frac{1-x}{5}$.
 So
\begin{eqnarray}\label{equ}
  \lambda(L_{\mathcal{F}_{10}}(8))=\lambda(G) =x{5 \choose 2}\left( {1-x \over 5} \right)^2={1 \over 5} \cdot 2x \cdot (1-x)^2\leq{1 \over 5}\left( {2 \over 3} \right)^3={8 \over 135}.
 \end{eqnarray}
 So $\lambda(\mathcal{F}_{10})\le {1 \over 4}\lambda(L_{\mathcal{F}_{10}}(8))\le {2 \over 135}$.
\qed
\end{proof}
\medskip

Case 11. $2357, 2348, 2456 \notin \mathcal{F}$ and $2356 \in \mathcal{F}$.

In this case, we have $1478 \notin \mathcal{F}$. Then
 \[
 \begin{split}
 \mathcal{F} \subseteq \mathcal{F}_{11}= \{12i_1j_1, 13i_2j_2, 145k, 146l, 156l, 2345, 2346,2347, 2356:3 \le i_1 < j_1 \le n,\\ 4 \le i_2 < j_2 \le n,6\le k \le n, 7\le l\le n\}.
 \end{split}
 \]
\begin{lemma}\label{F_{51}}
$\lambda(\mathcal{F}_{11})\le \frac{2}{135}<0.0169$.
\end{lemma}
\begin{proof}
 Note that $L_{\mathcal{F}_{11}}(8)=\{12i, 13j, 145, 146, 156: 3 \le i \le n, 4 \le j \le n,i,j\neq8\}.$
Since the pair of vertices $\{6,7\}$ is not covered in $L_{\mathcal{F}_{11}}(8)$, then by Lemma \ref{lemma-CP}, $\lambda(L_{\mathcal{F}_{11}}(8))= \lambda(\{1ij:2\leq i<j\leq 6\})$. So 
 So $\lambda(\mathcal{F}_{11})\le {1 \over 4}\lambda(L_{\mathcal{F}_{11}}(8))\le {2 \over 135}$ from (\ref{equ}).
\qed
\end{proof}
\medskip

Case 12. $2348, 2356 \notin \mathcal{F}$ and $2347 \in \mathcal{F}$.

In this case, we have $1568 \notin \mathcal{F}$. Then
\[
\begin{split}
\mathcal{F} \subseteq \mathcal{F}_{12}= \{12i_1j_1, 13i_2j_2, 14i_3j_3, 1567, 2345, 2346, 2347: 3 \le i_1 < j_1 \le n,\\ 4 \le i_2 < j_2 \le n, 5 \le i_3 < j_3 \le n\}.
 \end{split}
 \]
\begin{lemma}\label{F_{52}}
$\lambda(\mathcal{F}_{12})\le \frac{1}{72}<0.0169.$
\end{lemma}
\begin{proof}
Let $\vec{x}=(x_1,\dots,x_n)$ be an optimum weighting of $\mathcal{F}_{12}$. Note that $L_{\mathcal{F}_{12}}(8)=\{12i, 13j, 14k: 3 \le i \le n, 4 \le j \le n,5 \le k \le n,i,j,k\neq8\}.$
Since the pair of vertices $\{5,6\}$ is not covered in $L_{\mathcal{F}_{12}}(8)$, then by Lemma \ref{lemma-CP}, $\lambda(L_{\mathcal{F}_{12}}(8))= \lambda(\{1ij:2\leq i<j\leq 5\})$. Similar to (\ref{equ}), $\lambda(L_{\mathcal{F}_{12}}(8))\leq\frac{1}{18}.$ Then we have $\lambda(\mathcal{F}_{12})\le {1 \over 72}$.
\qed
\end{proof}
\medskip

Case 13. $2356, 2347 \notin \mathcal{F}$ and $2346 \in \mathcal{F}$.

In this case, we have $1578 \notin \mathcal{F}$. Then
\[
\begin{split}
\mathcal{F} \subseteq \mathcal{F}_{13}= \{12i_1j_1, 13i_2j_2, 14i_3j_3, 156k, 2345, 2346 : 3\le i_1 < j_1 \le n,4 \le i_2 < j_2 \le n,\\5 \le i_3 < j_3 \le n,7\leq k\leq n\}.
\end{split}
\]
\begin{lemma}\label{F_{61}}
$\lambda(\mathcal{F}_{13})\le \frac{2}{135}<0.0169$.
\end{lemma}
\begin{proof}
Note that $L_{\mathcal{F}_{13}}(8)=\{12i, 13j, 14k, 156: 3 \le i \le n, 4 \le j \le n, 5 \le k \le n,i,j,k\neq8\}$.
Since the pair of vertices $\{6,7\}$ is not covered in $L_{\mathcal{F}_{13}}(8)$, then by Lemma \ref{lemma-CP}, $\lambda(L_{\mathcal{F}_{13}}(8))= \lambda(\{1ij:2\leq i<j\leq 6\})$. So $\lambda(\mathcal{F}_{13})\le {1 \over 4}\lambda(L_{\mathcal{F}_{13}}(8))\le {2 \over 135}$ from (\ref{equ}).
\end{proof}
\medskip

Case 14. $2346 \notin \mathcal{F}$ and $2345 \in \mathcal{F}$.

In this case, we have $1678 \notin \mathcal{F}$. Then
\[
\begin{split}
\mathcal{F} \subseteq \mathcal{F}_{14}=
 \{12i_1j_1, 13i_2j_2, 14i_3j_3, 15i_4j_4, 2345:3\le i_1 < j_1 \le n,4 \le i_2 < j_2 \le n,\\5 \le i_3 < j_3 \le n,6 \le i_4 < j_4 \le n,\}.
\end{split}
\]
\begin{lemma}\label{F_{7}}
$\lambda(\mathcal{F}_{14})\le \frac{2}{135}<0.0169$.
\end{lemma}
\begin{proof}
Note that $L_{\mathcal{F}_{14}}(8)=\{12i, 13j, 14k, 15l: 3 \le i \le n, 4 \le j \le n,5 \le k \le n,6 \le l \le n,i,j,k,l\neq8\}$.
Since the pair of vertices $\{6,7\}$ is not covered in $L_{\mathcal{F}_{14}}(8)$, then by Lemma \ref{lemma-CP}, $\lambda(L_{\mathcal{F}_{14}}(8))= \lambda(\{1ij:2\leq i<j\leq 6\})$. So $\lambda(\mathcal{F}_{14})\le {1 \over 4}\lambda(L_{\mathcal{F}_{14}}(8))\le {2 \over 135}$ from (\ref{equ}).
\end{proof}
\medskip

Case 15. $2345 \notin \mathcal{F}$.

 In this case, $ \mathcal{F} \subseteq \mathcal{F}_{15}£º= \{1ijk:2 \le i<j<k\le n\}=\mathcal{S}_n.$ By Lemma \ref{F_0} we have $\lambda(\mathcal{F}_{15})= {9 \over 512}{(n-2)(n-3) \over (n-1)^2}\le {9 \over 512}$.

 If $\mathcal{F} \nsubseteq \mathcal{S}_n$,  then by the argument in Cases 1-14,  $\mathcal{F}$ is a subgraph of one of the $K_7^4$ and  $\mathcal{F}_i, i \in [14]$. So  $\lambda(\mathcal{F})<0.0169$.
  Hence we complete the proof.
\qed

\section{An application to a hypergraph Tur\'an problem}

Given $r$-graphs $F$ and $G$, a function $f: V(F)\longrightarrow V(G)$ is a {\em homomorphism} if it preserves edges, i.e., $f(i_1)f(i_2)\dots f(i_r) \in G$ if $i_1i_2\dots i_r \in F$. We say $G$ is {\em $F$-hom-free} if there is no homomorphism from $F$ to $G$.
We need the following connection between Tur\'an densities and Lagrangians due to Sidorenko.

\begin{lemma}{\em (see e.g. [\cite{Keevash}, Section 3])} \label{lemmahom}
Given an $r$-graph $F$, $\pi(F)$ is the supremum of $r! \lambda(G)$ over  all dense $F$-hom-free $r$-graphs $G$.
\end{lemma}
Write $H_8^{M_2^4}$ as $K_{4,4}^4$.
From Theorem \ref{theoremM^4_2} and Lemma \ref{lemmahom} we get that
\begin{theo}\label{turannumber1}
$\pi(K_{4,4}^4)={27 \over 64}$.
\end{theo}
\begin{proof}
Since $\mathcal{S}$ is $K_{4,4}^4$-free, we get the lower bound. For the upper bound, by Lemma \ref{lemmahom}, it suffices to show that
$\lambda(G) \le {9 \over 512} $ for any  dense $K_{4,4}^4$-free $4$-graph $G$. For every  dense $K_{4,4}^4$-hom-free $4$-graph $G$, we claim that $G$ is $M_2^4$-free. Otherwise suppose there are two disjoint edges $e, f \in G$. Since
 $G$ cover pairs, then there is an edge $e_{ab} \in G$ with $\{a, b\} \subseteq e_{ab}$ for every $a \in e$ and $b \in f$.
 Thus $\{e, f\} \cup \{e_{ab}:a\in e, b\in f\}$ forms a copy of $K_{4,4}^4$ which
 contradicts $G$ being $K_{4,4}^4$-hom-free. Then $\lambda(G) \le {9 \over 512} $ by Theorem \ref{theoremM^4_2}.
\qed
\end{proof}

\begin{coro}
$ex(n, K_{4,4}^4)={9 \over 512}n^4+o(n^4)$ for sufficiently large $n$.
\end{coro}

\subsection{Stability}

In order to prove Theorem \ref{theorem}, we will first prove the following stability result.

\begin{theo}\label{theostability}
For any $\varepsilon >0$, there exists $\delta >0$ and an integer $n_0$ such that if $\mathcal{F}$ is a $K_{4,4}^4$-free $4$-graph with $n\ge n_0$ vertices and at least $({9 \over 512}-\delta)n^4 $ edges, then there exists a partition $V(\mathcal{F})=A \cup B $ of the vertex set of $\mathcal{F}$ such that $|\{e \in \mathcal{F}: |e \cap A| \ge 2\}|+|\{e \in \mathcal{F}: e \subseteq B\}| < \varepsilon n^4 $.
\end{theo}

We first show that it suffices to prove the result under the assumption that $\mathcal{F}$ is $K_{4,4}^4$-hom-free.
\begin{remark}
Given an $r$-graph $F$ and $p \ge |V(F)|$, $G$ is $H_p^F$-hom-free if and only if $G$ is ${\mathcal{K}}_p^F$-free.
\end{remark}
\begin{proof}
We first prove that if $G$ is $H_p^F$-hom-free then $G$ is ${\mathcal{K}}_p^F$-free. Otherwise suppose that $G$ contains a member of ${\mathcal{K}}_p^F$, say $K$.
By the definition of ${\mathcal{K}}_p^F$, $K$ contains a core $C$ of size $p$ such that $K[C]$ contains $F$ as a subgraph.
We now construct a map $f$ from $V(H_p^F)$ to $V(K)$. Map the core of $H_p^F$ to $C$, that is the core of $K$, such that $F$ in $H_p^F$ is mapped to a copy of $F$ in $K$.
For other vertices in $V(H_p^F)$, e.g. $B_{ij}$ such that $\{i,j\}\cup B_{ij} \in H_p^F$,
where $i,j$ are in the core of $H_p^F$, maps $B_{ij}$ to $B'_{ij}$, where $\{f(i),f(j)\}\cup B'_{ij} \in K$ (since each pair in $C$ is covered by an edge in $K$, $\{f(i),f(j)\}$ is contained in some edge of $K$). So $f$ is a homomorphism from $H_p^F$ to $K$, a contradiction.

Now we prove that if $G$ is ${\mathcal{K}}_p^F$-free then $G$ is $H_p^F$-hom-free.
Otherwise suppose that $G$ is not $H_p^F$-hom-free, that is, there is a homomorphism $g$ from $H_p^F$ to $G$.
Denote the core of $H_p^F$ as $C'$ and $C''=\{g(v):v\in C'\}$.
We first show that $|C''|=|C|=p$. Otherwise suppose that $|C''|<p$, then $\exists u,v \in C$ such that $g(u)=g(v)$, which contradicts that $g$ is a homomorphism, since $u, v$ is contained in some edges of $H_p^F$.
This implies that $F \subseteq G[C'']$. For every pair $x,y \in C$ that is not covered by $F$, fix one $(r-2)$-set as $B_{xy}$ such that $\{x, y\} \cup B_{xy} \in  H_p^F$.
Then  $\{g(w): w\in \{x, y\} \cup B_{xy}\}$ is an edge of $G$ for all pairs $x,y \in C$ that is not covered by $F$, since $g$ is a homomorphism.
Hence $\{\{g(v):v\in e\}: e\in H_p^F\}$ is a member of ${\mathcal{K}}_p^F$, a contradiction.
\qed
\end{proof}
\medskip

In Section 7 of \cite{BIJ}, Brandt-Irwin-Jiang proved that every $H_p^F$-free $r$-graph on $[n]$ can be made ${\mathcal{K}}_p^F$-free by removing $\Theta(n^{r-1})$ edges. Let $F=M_2^4$ and $p=8$, then it suffices to prove Theorem \ref{theostability} under the assumption that $\mathcal{F}$ is $K_{4,4}^4$-hom-free.

Part of our proof follows the approach in \cite{Pikhurko,HK} by Pikhurko and Hefetz-Keevash. We gradually adjust $\mathcal{F}$ by iterating a process which is called {\em Symmetrization}. This process consists of two parts: {\em Cleaning} is to delete vertices with `small' degree, and {\em Merging} is to replace the link of a vertex $v$ by the link of a vertex $u$ if $d(v)\le d(u)$ and the pair $u, v$ is not covered by an edge. It terminates if we can no longer clean any vertex. We show the terminating $4$-graph is isomorphic to $S^4(n')$ with $n'=(1-o(1))n$. Then we trace back and show that the symmetrization process is `stable' (does not change the $4$-graph much).

Now let us give the proof precisely. Clearly, we can also assume that $\varepsilon$ is sufficiently small and $\delta \ll \varepsilon$. Let $\alpha, \beta, \gamma$ and $\delta$ be real numbers satisfying $$1\gg \gamma \gg \beta \gg \alpha \gg \varepsilon \gg \delta \gg n_0^{-1}.$$

We operate the symmetrization process for a {\rm pointed} $4$-graph: by this we mean a triple $(\mathcal{G}, \mathcal{P}, U)$,
where $\mathcal{G}=(V,E)$ is a $4$-graph, $ \mathcal{P}=\{P_u:u\in U\}$ is a partition of $V$,
i.e. for any $v \in V$ there exists some $u \in U$ such that $v \in P_u$ and we give an order for vertices in $P_u$ for every $u \in U$,
and $U \subseteq V$ is a transversal of $ \mathcal{P}$ and every $u \in U$ is the representative of $P_u$. Let us describe the process precisely.
\medskip

\noindent{\bf Cleaning:}
\newline
\noindent{\bf Input:} {\em A pointed $4$-graph $(\mathcal{G},\mathcal{P},U)$ on $n$ vertices. }
\newline
\noindent{\bf Output:} {\em A pointed $4$-graph $(\mathcal{G}',\mathcal{P}',U')$ on $n' \le n$ vertices.}
\newline
\noindent{\bf Process:} {\rm If  $\delta(\mathcal{G}) \ge (9/128-\alpha)n^3$ or $V(\mathcal{G})=\emptyset$ then stop and return $(\mathcal{G}',\mathcal{P}',U')=(\mathcal{G},\mathcal{P},U)$, where $n' = |V(\mathcal{G}')|$. Otherwise, let $u \in U$ be an arbitrary vertex such that $d_{\mathcal{G}}(u)< (9/128-\alpha)n^3$. If $P_u=\{u\}$ then apply cleaning to $(\mathcal{G}- \{u\}, \mathcal{P}-\{P_u\},U-\{u\})$.
Otherwise let $v \in P_u$ be the vertex with the maximum order and apply cleaning to $(\mathcal{G}- \{v\},(\mathcal{P}-\{P_u\})\cup \{P_u-\{v\}\},U)$, where $n' = |V(\mathcal{G}')| <n$. }

Note that this algorithm will always terminate and $\delta(\mathcal{G}') \ge (9/128-\alpha)n'^3$ or $\mathcal{G}'$ is empty. A vertex set $U$ is {\em covered} by $\mathcal{G}$ if for every pair $u,v \in U$, $\{u,v\}$ is contained in an edge of $\mathcal{G}$.
\medskip

\noindent{\bf Merging:}
\newline
\noindent{\bf Input:} {\em A pointed $4$-graph $(\mathcal{G},\mathcal{P},U)$ on $n$ vertices. }
\newline
\noindent{\bf Output:} {\em  A pointed $4$-graph $(\mathcal{G}',\mathcal{P}',U')$ on the same vertex set of $\mathcal{G}$.}
\newline
{\bf Process:} {\rm If $U$ is covered by  $\mathcal{G}$ then stop and return $(\mathcal{G}',\mathcal{P}',U')=(\mathcal{G},\mathcal{P},U)$. Otherwise, let $u, v \in U$ be arbitrary vertices such that $\{u,v\}$ is not covered.
Assume that $d_{\mathcal{G}}(v) \le d_{\mathcal{G}}(u)$. Merge $P_v$ into $P_u$ (clone), that is, let $P'_u=P_u \cup P_v$  and $P'_w=P_w$ for every $w \in U\setminus \{u, v\}$.
Moreover, let $U'=U\setminus \{v\}$ and let $\mathcal{G}'$ be a blowup of $\mathcal{G'}[U']$ with partition sets $\{P'_w:w\in U'\}$, i.e., $E(\mathcal{G}')=\bigcup_{e\in \mathcal{G}[U']} \prod_{u\in e} P'_u$.
Suppose the order of $P_u$ is $u_1 \prec u_2 \prec \dots \prec u_k$ and the order of $P_v$ is $v_1 \prec v_2 \prec \dots \prec v_t$, then we give $P'_w$ an order as $u_1 \prec \dots \prec u_k \prec v_1 \prec \dots \prec v_t$.
Let $\mathcal{P}'=\{P'_w:w\in U'\}$ and return $(\mathcal{G}',\mathcal{P}',U')$.
}

Clearly, $ \mathcal{P}'=\{P'_u:u\in U'\}$ is a partition of $V(\mathcal{G}')$ satisfying for any $v \in V(\mathcal{G}')$ there exists some $u \in U'$ such that $v \in P'_u$, and $U' \subseteq V(\mathcal{G}')$ is a transversal of $ \mathcal{P}'$. Note that the merging processes at most one step and for every $u \in U$, $\forall v,w\in P_{u}$, we have $v \sim w$.
Now we are ready to describe the symmetrization process:
\medskip

\noindent{\bf Symmetrization:}
\newline
\noindent{\bf Input:} {\rm A $K_{4,4}^4$-free $4$-graph $\mathcal{F}=(V,E)$.}
\newline
\noindent{\bf Output:} {\rm A pointed $4$-graph $(\mathcal{F}^*,\mathcal{P},U)$.}
\newline
\noindent{\bf Initiation:} {\rm  Let $\mathcal{H}_0=\mathcal{F}=(V,E)$, $\mathcal{P}_0=\{P_{0,w}:w\in U_0\}$, $U_0=V$, where $P_{0,w}=\{w\}$ for every $w \in V$. And the order of every part $P_{0,w}$ which has only one single vertex is trivial. Set $i=0$.}
\newline
\noindent{\bf Iteration:}
{\rm Apply Cleaning to $(\mathcal{H}_i, \mathcal{P}_i, U_i)$ and let $\left(\mathcal{H}'_{i+1}, \mathcal{P}'_{i+1}, U'_{i+1}\right)$ be the output, where $\mathcal{H}'_{i+1}=(V'_{i+1},E'_{i+1})$, and $\mathcal{P}'_{i+1}=\{P'_{i+1,w}:w\in U'_{i+1}\}$.
 Apply Merging to $(\mathcal{H}'_{i+1}=(V'_{i+1},E'_{i+1}), \mathcal{P}'_{i+1}=\{P'_{i+1,w}:w\in U'_{i+1}\}, U'_{i+1})$ and let $(\mathcal{H}_{i+1}=(V_{i+1},E_{i+1}), \mathcal{P}_{i+1}=\{P_{i+1,w}:w\in U_{i+1}\}, U_{i+1})$ be the output.
If $(\mathcal{H}_{i+1}, \mathcal{P}_{i+1}, U_{i+1})=(\mathcal{H}'_{i+1}, \mathcal{P}'_{i+1}, U'_{i+1})$, then stop and return $(\mathcal{F}^*,\mathcal{P},U)=(\mathcal{H}_i, \mathcal{P}_i, U_i)$. Otherwise, increase $i$ by one and repeat Cleaning and Merging.
}
\medskip

Let $(\mathcal{F}^*, \mathcal{P}, U)$ be the output of applying Symmetrization to $\mathcal{F}$. Let
$$\mathcal{H}_0=\mathcal{F},\mathcal{H}'_1,\mathcal{H}_1,\dots,\mathcal{H}'_t,\mathcal{H}_t=\mathcal{F}^*$$
be the sequence of $4$-graphs produced during this process, where $\mathcal{H}'_i=(V'_i,E'_i)$ and $\mathcal{H}_i=(V_i, E_i)$ for every $1 \le i \le t$.

We split the proof into two stages. In the first stage we show that $\mathcal{F}^*$ is contained in a large blowup of $\mathcal{S}_k$, where $k\approx 3n/4$.
In the second stage we show that $\mathcal{F}[V_t]$ is a subgraph of a blowup of $\mathcal{S}_k$ (Just clone the vertex that all edges of $\mathcal{S}_k$ intersect for about $n/4$ copies). Then Theorem \ref{theostability} will follow easily from this.

We start with the first stage, which we prove by a series of lemmas.

\begin{lemma}\label{lemmas1}
The following properties hold for every $0 \le i \le t$.
\\{\rm (1)} For every $0\le j \le i$ the set $U_j\cap V_i$ is a transversal for the partition $\{P_{j,v}\cap V_i: v \in U_j\cap V_i\}$.
\\{\rm(2)}  $\mathcal{H}_i[U_i]= \mathcal{F}[U_i]$.
\\{\rm(3)} $|e \cap P_{j,v}|\le 1$ for every $e\in E_i$ and every $v \in V_i$.
\\{\rm(4)} For every $u \in U_i$, $\forall v,w\in P_{i,u}$, $v \sim w$.
\\{\rm(5)} For every $i\ge 1$, $U_i \subseteq U_{i-1}$ and $V_i \subseteq V_{i-1}$.
\end{lemma}
\begin{proof}
By the process of symmetrization algorithm, properties (2)-(5) hold obviously.
Now we prove property (1). It is sufficient to show that $\{P_{j,v}\cap V_i: v \in U_j\cap V_i\}$ is a partition of $V_i$.
(i) For any distinct vertices $u,v \in U_j\cap V_i$, $(P_{j,u}\cap V_i)\cap (P_{j,v}\cap V_i)\subseteq P_{j,u}\cap P_{j,v}=\emptyset$.
(ii) Since $V_i \subseteq V_j$ by property (5) and $V_j=\bigcup_{w \in U_j}P_{j,w}$ (a partition), then $\forall v \in V_i$, $\exists u \in U_j$ such that $v \in P_{j,u}\cap V_i$. On the other hand, $\bigcup_{u \in U_j\cap V_i}(P_{j,u}\cap V_i)=(\bigcup_{u \in U_j\cap V_i}P_{j,u})\cap V_i\subseteq V_j\cap V_i=V_i$. So $\bigcup_{u \in U_j\cap V_i}(P_{j,u}\cap V_i)=V_i$
\qed
\end{proof}

\begin{lemma}\label{lemmas2}
Let $1 \le i\le t$ and suppose that $P'_{i,v}$ was merged into $P'_{i,u}$ during the $i$th merging step. Then
$P'_{i,u} \cap V(\mathcal{F}^*)=\emptyset$ implies $P'_{i,v} \cap V(\mathcal{F}^*)=\emptyset$.
\end{lemma}
\begin{proof}
Suppose $P'_{i,u} \cap V(\mathcal{F}^*)=\emptyset$, this implies that $u$ has been cleaned in some $j$th merging step, where $i <j \le t$. Since $u$ is the last vertex among $P'_{i,v}\cup P'_{i,u}$ that be cleaned, so all vertices in $P'_{i,v}$ have been cleaned. Therefore, $P'_{i,v} \cap V(\mathcal{F}^*)=\emptyset$.
\qed
\end{proof}
\medskip

Since merging preserves the property of $K_{4,4}^4$-hom-freeness and deleting vertices certainly also does, we get the following lemma.
\begin{lemma}\label{lemmas3}
$\mathcal{H}_i$ is $K_{4,4}^4$-hom-free for every $0 \le i \le t$.
\end{lemma}

The next Lemma asserts that merging does not decrease the number of edges.
\begin{lemma}\label{lemmaM}
$e(\mathcal{H}_i) \ge e(\mathcal{H}'_i)$ holds for  every $1 \le i \le t$.
\end{lemma}

\begin{lemma}\label{lemmas1}
$\mathcal{F}[U_t]$ is $M_2^4$-free.
\end{lemma}
\begin{proof}
Assume for the sake of contradiction that there exist two disjoint edges $e,f$ in $\mathcal{F}[U_t]$. Note that $\mathcal{F}[U_t]=\mathcal{F^*}[U_t]$ and $\mathcal{F}[U_t]$ cover pairs, then $\forall a\in e$ and $\forall b \in f$, $\exists e_{ab}\in \mathcal{F}[U_t]$ such that $a,b \in e_{ab}$. Hence $\{e,f\} \cup \{e_{ab}:a \in e, b\in f\} \subseteq \mathcal{F}[U_t] \subseteq \mathcal{F}$, which contradicts that $\mathcal{F}$ is $K_{4,4}^4$-hom-free.
\qed
\end{proof}

The following proposition follows immediately from the definition and is implicit in many papers (see \cite{Keevash} for instance). We include a short proof of it for completeness.
\begin{prop}\label{propblowup}
Let $F$ be an $r$-graph. Let $L$ be an $F$-free $r$-graph. Let $H$ be a blow-up of $L$ with $n$ vertices. Then $|H|\le {\pi_{\lambda}(F) \over r!}n^r$. In particular, $|\mathcal{F}^*|\le {9 \over 512} n^4$.
\end{prop}
\begin{proof}
Let $\vec{x}=(x_1,x_2,\dots,x_n)$ be a weighting of $H$, where $x_i={1 \over n}$ for all $i\in [n]$. Then
$${|H| \over n^r}= \lambda(H,\vec{x}) \le \lambda(H)=\lambda(L)\le \pi_{\lambda}(F)/r!,$$
where $\lambda(H)=\lambda(L)$ is from Corollary \ref{coroblowup}.
Therefore $|H|\le {\pi_{\lambda}(F) \over r!}n^r$.
Since $\mathcal{F}^*$ is a blowup of $\mathcal{F}^*[U_t]$ and $\mathcal{F}^*[U_t]$ is an $M_2^4$-free $4$-graph by Lemma \ref{lemmas1}, then
$|\mathcal{F}^*|\le {\pi_{\lambda}(M_2^4)n^4 \over 4!} = {9 \over 512} n^4$.
\qed
\end{proof}
\medskip

Let $C_i$ be the vertex set deleted by the $i$th cleaning, where $0 \le i \le t-1$, i.e. $C_i=V_{i}\setminus V_{i+1}'$. Let $C=\bigcup_{i=0}^{t-1}C_i$, so $C= V\setminus V_t$ is the set of vertices removed by the symmetrization algorithm.
The next lemma asserts that the symmetrization process does not delete too many vertices.
\begin{lemma}\label{lemmav}
$|V(\mathcal{F}^*)| \ge (1-\alpha)n$.
\end{lemma}
\begin{proof}
Write $s=|C|$. By Lemma \ref{lemmaM} and the definition of cleaning we have
\begin{eqnarray*}
|\mathcal{F}^*| &\ge& |\mathcal{F}|-\sum_{i=0}^{s-1}\left(9/128-\alpha\right)(n-i)^3 \\
&\ge& \left(9/512-\delta\right)n^4 - (\left(9/512-\alpha/4\right)(n^4-(n-s)^4)+\delta n^4)
\end{eqnarray*}
since
$$\sum_{i=0}^{s-1}(n-i)^3=\sum_{i=1}^{n}i^3-\sum_{i=1}^{n-s}i^3={1\over 4}n^2(n+1)^2-{1\over 4}(n-s)^2(n-s+1)^2\ge {1\over 4}n^4-{1\over 4}(n-s)^4-\delta n^4.$$
On the other hand, since $\mathcal{F}^*$ is $K_{4,4}^4$-hom-free, then by Proposition \ref{propblowup} we have
$$|\mathcal{F}^*| \le 9/512(n-s)^4.$$
This yields
$$ {\alpha \over 4}(n-s)^4 \ge ({\alpha \over 4} -2\delta)n^4. $$
Hence
$$ \left({n-s \over n}\right)^4 \ge {\alpha/4-2\delta \over \alpha/4} > 1-\alpha$$
for $\delta \ll \alpha$. Hence
$${n-s \over n} > (1-\alpha)^{1/4} \ge 1-\alpha.$$
So $s < \alpha n$ and $|V(\mathcal{F}^*)| \ge (1-\alpha)n$.
\qed
\end{proof}
\medskip

\begin{lemma}\label{lemmasf}
$\mathcal{F}^*[U_t]\subseteq \mathcal{S}_{|U_t|}$. Let $1$ be the vertex in $U_t$ that intersects all edges of $\mathcal{F}^*[U_t]$, then $|P_{t,1}|=({1 \over 4}\pm \beta)|V(\mathcal{F}^*)|$.
\end{lemma}
\begin{proof}
Suppose $U_t=[m]$. Let $\vec{x}=\{x_1,\dots,x_m\}$ be a weighting of $\mathcal{F}^*[U_t]$ such that $x_i= {|P_{t,i}| \over |V_t|}$ for every $i\in [m]$. So $x_i\ge 0$ and  $\sum_{i=1}^mx_i=1$. Then
\begin{eqnarray*}
\lambda(\mathcal{F}^*[U_t],\vec{x})
= \sum_{e \in \mathcal{F}^*[U_t]}\prod\limits_{i\in e}{|P_{t,i}| \over |V_t|}
= {1 \over |V_t|^4} e(\mathcal{F}^*).
\end{eqnarray*}
Since $d_{\mathcal{F}^*}(x)\ge ({9 \over 128}-\alpha)|V_t|^3$ for every $x \in V(\mathcal{F}^*)$,
so $e(\mathcal{F}^*) \ge ({9 \over 512}-{\alpha \over 4})|V_t|^4$. Then
$\lambda(\mathcal{F}^*[U_t],\vec{x})= {1 \over |V_t|^4} e(\mathcal{F}^*)\ge {9 \over 512}-{\alpha \over 4}$. Since $\alpha$ is small enough, then by Theorem \ref{theoremM^4_2},
 we have $\mathcal{F}^*[U_t]\subseteq \mathcal{S}_{|U_t|}$. This proves the first part.

For the second part, suppose that $|P_{t,1}|\neq({1 \over 4}\pm \beta)|V(\mathcal{F}^*)|$. Denote $A=P_{t,1}$ and $B=V(\mathcal{F}^*)\setminus A$.

{\em Case 1.} $|A|>({1 \over 4}+ \beta)|V(\mathcal{F}^*)|$. Then $|B|<({3 \over 4}- \beta)|V(\mathcal{F}^*)|$.
Then for any $v \in A$, $d_{\mathcal{F}^*}(v)\le {({3 \over 4}- \beta)|V(\mathcal{F}^*)| \choose 3}\le \left({9 \over 128}-{9 \over 33}\beta\right)|V(\mathcal{F}^*)|^3$, which contradicts that $\delta(\mathcal{F}^*)\ge \left({9 \over 128}-\delta\right)|V(\mathcal{F}^*)|^3$.

{\em Case 2.} $|A|<({1 \over 4}- \beta)|V(\mathcal{F}^*)|$. Suppose that $|A|=({1 \over 4}- \mu)|V(\mathcal{F}^*)|$,
so $|B|=({3 \over 4}+ \mu)|V(\mathcal{F}^*)|$, where $\mu > \beta$. Then for any $v \in B$, $d_{\mathcal{F}^*}(v)\le ({1 \over 4}- \mu)|V(\mathcal{F}^*|\cdot{({3 \over 4}+\mu )|V(\mathcal{F}^*)| \choose 2} = {1\over 2}\left({9 \over 64}-{3 \over 16}\mu-{5 \over 4}\mu^2-\mu^3\right)|V(\mathcal{F}^*)|^3
\le \left({9 \over 128}-{3 \over 16}\beta\right)|V(\mathcal{F}^*)|^3$, a contradiction.
\qed
\end{proof}

\medskip

This completes the first stage of the proof. The second stage is to show that $\mathcal{F}[V_t]$ is a subgraph of a blowup of $\mathcal{S}_{|U_t|}$, that is, $\mathcal{F}[V_t] \subseteq A'\times {B' \choose 3}$ for some $\{A', B'\}$ which is a partition of $V_t$ satisfying $|A'|\approx |B'|/3$.
To do so, we will trace back the Merging steps performed during symmetrization.

Recall that $1$ is the vertex in $U_t$ that intersects all edges of $\mathcal{F}^*[U_t]$. Let $A$ be the set $P_{t,1}$ and $B=V(\mathcal{F}_t)\setminus A$.
By Lemma \ref{lemmasf} we have $V_t=A\cup B$ with $|A|\thickapprox {1 \over 3}|B|$.  For every $0 \le i \le t$ we will find a partition $\mathcal{Q}_i=\{Q_{i,j}: j \in [2]\}$ of $V_t$ which satisfies the following properties:
\\(P1) $1 \in Q_{i,1}$;
\\(P2) For every $v \in U_t$, $P_{i,v}\cap V_t \subseteq Q_{i,1}$ or $P_{i,v}\cap V_t \subseteq Q_{i,2}$;
\\(P3) For every $e \in \mathcal{H}_i[V_t]$ we have $|e\cap Q_{i,1}|=1$ ( so $|e\cap Q_{i,2}|=3$), that is, $\mathcal{H}_i[V_t] \subseteq Q_{i,1}\times {Q_{i,2} \choose 3}$.
\medskip

Set $Q_{t,1}=A$ and $Q_{t,2}=B$. It follows by Lemma \ref{lemmasf} that $\mathcal{Q}_t=\{Q_{t,1}, Q_{t,2}\}$ satisfies (P1)-(P3).
Assume that we have already found a partition $\mathcal{Q}_i$ which satisfies (P1)-(P3) for some $i \in [t]$, now we will find a partition $\mathcal{Q}_{i-1}$ with the desired properties.

 Write
$$m=|V_t|, \:\: \mathcal{G}_i=\mathcal{H}_i[V_t], \:\: \mathcal{G}_{i-1}=\mathcal{H}_{i-1}[V_t] \:\: {\rm and} \:\: \mathcal{B}_i=Q_{i,1}\times {Q_{i,2} \choose 3}.$$
 We first prove the following Lemma which is vital for finding $\mathcal{Q}_{i-1}$ with the desired properties.
\begin{lemma}\label{lemmaspliting}
Let $0\le i\le t$.
\\{\em (i)} $\delta(\mathcal{G}_i)\ge ({9 \over 128}-2\alpha)m^3$, so $e(\mathcal{G}_i)\ge ({9 \over 512}-{1 \over 2}\alpha)m^4$.

 If a partition $\mathcal{Q}_i=\{Q_{i,1}, Q_{i,2}\}$ of $V_t$ satisfies (P1)-(P3), then
\\{\em (ii)} $|Q_{i,1}|=({1\over 4}\pm \beta)m$, so $|Q_{i,2}|=({3\over 4}\pm \beta)m$ and
\\{\em (iii)} $d_{\mathcal{B}_i \setminus \mathcal{G}_i}(x)\le \gamma m^3$ for all $x \in V_t$.
\end{lemma}
\begin{proof}
By definition of cleaning, $d_{\mathcal{H}_i}(x) \ge ({9 \over 128}-\alpha)m^3$ holds for every $x \in V_t$. It follows by Lemma \ref{lemmav} that
$$\delta(\mathcal{G}_i)\ge \left({9 \over 128}-\alpha \right)m^3-|V_i\setminus V_t|{n \choose 2} \ge \left({9 \over 128}-\alpha \right) m^3-{1 \over 2} \alpha \left({m \over 1-\alpha} \right)^3 \ge \left({9 \over 128}-2\alpha \right)m^3.$$
This proves (i).

Next, assume for the sake of contradiction that $|Q_{i,1}|\neq \left({1 \over 4}\pm \beta\right)m$.
First suppose that $|Q_{i,1}|>\left({1 \over 4}+ \beta\right)m$. Then $|Q_{i,2}|<({3 \over 4}- \beta)m$.
For any $v \in Q_{i,1}$, $d_{\mathcal{F}^*}(v)\le {({3 \over 4}- \beta)m \choose 3}\le \left({9 \over 128}-{9 \over 33}\beta\right)m^3$, which contradicts to $\delta(\mathcal{F}^*)\ge \left({9 \over 128}-\delta\right)m^3$.
Now assume that $|Q_{i,1}|<\left({1 \over 4}- \beta\right)m$. Suppose that $|Q_{i,1}|=({1 \over 4}- \mu)m$, where $\beta \le \mu \le 1/4$ is a real number.
So $|Q_{i,2}|=({3 \over 4}+ \mu)m$. For any $v \in Q_{i,2}$, $d_{\mathcal{F}^*}(v)\le ({1 \over 4}- \mu)|V(\mathcal{F}^*|\cdot{\left({3 \over 4}+\mu \right)m \choose 2} = {1\over 2}\left({9 \over 64}-{3 \over 16}\mu-{5 \over 4}\mu^2-\mu^3\right)m^3
\le \left({9 \over 128}-{3 \over 16}\beta\right)m^3$, a contradiction.
This proves (ii).

Finally, let $x \in Q_{i,1}$ and $y \in Q_{i,2}$ be two arbitrary vertices. By (P3) and (ii) we have
$$d_{\mathcal{B}_i}(x)\le {\left({3 \over 4}+ \beta\right)m \choose 3} \le \left({9 \over 128}+{\gamma \over 2}\right)m^3$$
and
$$d_{\mathcal{B}_i}(y)\le ({1 \over 4}+ \beta)m{({3 \over 4}+ \beta)m \choose 2} \le ({9 \over 128}+{\gamma \over 2})m^3.$$
Since $\mathcal{G}_i \subseteq \mathcal{B}_i$ by (P3) for $\mathcal{Q}_i$, this implies (iii), using (i) and $\alpha \ll \gamma$.
\qed
\end{proof}

\begin{lemma}\label{lemmagen}
Let $c$ be a real number satisfying $\gamma \ll c \le 10^{-2}$. Let $\mathcal{G}$ be a $K_{4,4}^4$-hom-free $4$-graph with vertex set $Q_{i,1} \cup Q_{i,2}$, where $Q_{i,1} \cap Q_{i,2}=\emptyset$ and let $\mathcal{B}=Q_{i,1} \times {Q_{i,2} \choose 3}$.
Let $m=|Q_{i,1} \cup Q_{i,2}|$.
If there are vertex sets $A\subseteq Q_{i,1}$ and $B\subseteq Q_{i,2}$ satisfying that
\newline
{\em (1)} $|A|\ge cm$ and $|B|\ge cm$,
\newline
{\em(2)} $d_{\mathcal{G}}(x)\ge (9/128-2\alpha)m^3$ for every $x\in Q_{i,1} \cup Q_{i,2}$ and
\newline
{\em(3)} $d_{\mathcal{B}[A\cup B] \setminus \mathcal{G}[A\cup B]}(x)\le \gamma m^3$ for every $x\in A$,
\newline
then there is no edge $e$ of $\mathcal{G}$ such that $|e \cap A|\ge 2$.
\end{lemma}
\begin{proof}
Assume for the sake of contradiction that there exists one edge $e \in \mathcal{G}$ such that $|e \cap A|\ge 2$. Let $a, b \in e$ satisfy $a, b \in A$.
For $x=a$ or $b$ let
$$B(x):=\{x' \in B\setminus e: \exists e' \in \mathcal{G} \:\: {\rm such} \:\: {\rm that} \:\: \{x,x'\}\subseteq e'\}.$$
By condition (3) we have $ |B(x)|\ge |B|-cm/10$.
Let $Q'=B(a)\cap B(b)$. It is clear that $|Q'|\ge 3cm/4$.
Let $Q'_1, Q'_2, Q'_3$ be an arbitrary partition of $Q'$ satisfying $|Q'_j| \ge c m/4$ for every $j \in [3]$.
Fix  $D \subseteq A\setminus e$ satisfying $|D|\ge c m/2$.
Denote the maximum set of the disjoint triples in $Q'_1$ belonging to $L_{\mathcal{G}}(a)$ as $M_a$.
We claim that $|M_a|\ge c m/20$. Otherwise suppose that $|M_a|<c m/20$. Since for every triple $g\in Q'_1 \setminus (\cup_{f\in M_a}f)$, $g\notin L_{\mathcal{G}}(a)$.
Then there are at least ${|Q'_1 \setminus (\cup_{f\in M_a}f)| \choose 3}>{cm/10 \choose 3}$ triples in $Q'_1$  not belonging to $L_{\mathcal{G}}(a)$, which contradicts to (3).
Similarly, there are at least $cm/20$ pairwise disjoint triples in $Q'_2$, denoted as $M_b$, belonging to $L_{\mathcal{G}}(b)$.
Since $\mathcal{G}$ is $K_{4,4}^4$-hom-free, then $\forall f_1 \in M_a$, $\forall  f_2 \in M_b$, $\exists v_1\in f_1, v_2\in f_2$, $v_1v_2u_1u_2 \notin \mathcal{G}$ for all  $u_1 \in D$ and $u_2 \in Q'_3$.
Let $v$ be an arbitrary vertex in $D$. Then there are at least $ |M_a|\cdot |M_b| \cdot |Q'_3|$ triples not belonging to $L_{\mathcal{B}[A\cup B] \setminus \mathcal{G}[A\cup B]}(v)$, which contradicts to (3).
\qed
\end{proof}
\medskip

Let $u,v \in U'_i$ be such that in the $i$th Merging step $P'_{i,v}$ was merged into $P'_{i,u}$. Note that $u \in U_i$ and $v \notin U_i$. Denote
$$ A_u=P'_{i,u}\cap V_t  \:\: {\rm and} \:\:  A_v=P'_{i,v}\cap V_t.$$
We can view $\mathcal{G}_i$ as being obtained from $\mathcal{G}_{i-1}$ by Merging $A_v$ to $A_u$.
Since $\mathcal{Q}_i$ satisfies (P2), then $A_v \cup A_u \subseteq Q_{i,1}$ or $A_v \cup A_u \subseteq Q_{i,2}$.
In both cases, let $$W_1=Q_{i,1}\setminus A_v \:\: {\rm and} \:\: W_2=Q_{i,2}\setminus A_v.$$
Suppose the partition $\{W'_1, W'_2\}$ of $V_t$ is obtained by adding $A_v$ to $W_1$ or $W_2$ such that
$$\Sigma:=|\{e\in E_i: |e \cap W'_1|=2 \:\: {\rm or } \:\: e\subseteq W'_2\}|+2|\{e\in E_i: |e \cap W'_1|=3\}|+3|\{e\in E_i: e \subseteq W'_1\}|$$
is the smaller one.

Let $Q_{i-1,1}=W'_1$ and $Q_{i-1,2}=W'_2$ and let $\mathcal{Q}_{i-1}=\{Q_{i-1,1},Q_{i-1,2}\}$. We call an edge $e$ of $\mathcal{G}_{i-1}$ {\em bad} if $|e \cap W'_1|=2$ or $e \subseteq W'_2$, {\em very bad} if $|e \cap W'_1| = 3$, {\em worst} if $|e \cap W'_1| =4$ or {\em good} otherwise. We will prove that $\mathcal{Q}_{i-1}$ satisfies (P1)-(P3).
This is immediate for (P2), and (P1) follows since $W_1\neq \emptyset$ and $1 \notin A_v$ by the definition of Merging. It remains to prove (P3), i.e. all edges are good.
Equivalently, we need to show that $\Sigma=0$, as every bad edge is counted exactly once in $\Sigma$, every very bad edge is counted exactly  twice in $\Sigma$ and every worst edge is counted exactly three times in $\Sigma$, whereas good edges are not counted at all.

Note that any $e \in \mathcal{G}_{i-1}$ that is not good satisfies $|e \cap A_v|=1$. Since $|e \cap A_v|\le 1$ by the definition of Merging and $|e \cap A_v|\ge 1$  by (P3) for $\mathcal{Q}_i$.

We say that a vertex of $A_v$ is {\em bad} if it is contained in at least $10^{-3}m^3$ edges which are not good.
Before proving that all edges are good, we will prove that a vertex of $A_v$ cannot be contained in too many edges which are not good.

\begin{lemma}\label{lemmabad}
There are no bad vertices in $A_v$.
\end{lemma}
\begin{proof}
Assume for the sake of contradiction that $x \in A_v$ is a bad vertex. Then every vertex in $A_v$ is bad. We divide it into two cases according to $A_v \cup A_u \subseteq Q_{i,1}$ or $A_v \cup A_u \subseteq Q_{i,2}$.
\newline
{\em Case 1.} $A_v \cup A_u \subseteq Q_{i,1}$. Note that $W_1=Q_{i,1}\setminus A_v$ and $W_2=Q_{i,2}$ in this case.

Subcase 1. Adding $A_v$ to $W_1$ minimises $\Sigma$.
Since every bad edge (if there exist bad edges) intersects $A_v$,
so every bad edge intersect $W_1$. Fix an edge $xaa_1a_2 \in \mathcal{G}_{i-1}$ such that $a \in W_1$. Note that $a \notin A_v$.

Claim 1.1.1: There are at least ${1 \over 30}m^3$ good edges containing $x$. Otherwise we consider the partition $\{W''_1,W''_2\}$ of $V_t$ obtained by adding $A_v$ to $W_2$ rather than $W_1$.
This new partition implies that every good edge containing a vertex of $A_v$ turns bad, this contributes to $\Sigma$ at most ${1 \over 30}|A_v|m^3$.
Every bad edge turns good, every very bad edges turns bad and every worst edge turns very bad, so this reduces $\Sigma$ by at least $(9/128-2\alpha-1/30)|A_v|m^3> {1 \over 30}|A_v|m^3$, which contradicts the minimality of $\Sigma$.

Claim 1.1.2: There are at least $10^{-3}m$ pairwise disjoint triples in $W_2$ belonging to $L_{\mathcal{G}_{i-1}}(x)$.
Otherwise, since every such triple intersects with less than $3{m \choose 2}$ triples in $W_2$, then there are at most  $3{m \choose 2}\cdot 10^{-3}m<m^3/30$ good edges of $\mathcal{F}$ containing $x$, contradicting to Claim 1.1.1.

Fix a set of such triples of size $10^{-3}m$ as $M(x)$.
Let $$B(x):=\{x' \in W_2: \exists e \in \mathcal{G}_{i-1} \:\: {\rm such} \:\: {\rm that} \:\: \{x,x'\}\subseteq e\}.$$

Claim 1.1.3: $|B(x)\setminus (\cup_{e\in M(x)}e)|\ge m/10$. Since all good edges of $\mathcal{F}$ containing $x$ are contained in $\{x\}\times{B(x)\choose 3}$, then
${|B(x)|\choose 3}\ge {1 \over 30}m^3$ and Claim 1.1.3 follows.

Denote the maximum set of the disjoint triples in $B(x)\setminus (\bigcup_{e\in M(x)}e)$ belonging to $L_{\mathcal{G}_{i-1}}(a)$ (in this sense, they belong  to $L_{\mathcal{G}_{i}}(a)$ as well) as $M(a)$.

Claim 1.1.4: $|M(a)|\ge m/60$.
Otherwise suppose that $|M_a|<m/60$. Since for every triple $g$ in $B(x)\setminus (\bigcup_{f\in M(x)\cup M(a)}f)$, $g\notin L_{\mathcal{G}_{i}}(a)$ (or $L_{\mathcal{G}_{i-1}}(a)$).
Then there are at least ${m/20 \choose 3}$ triples in $W_2$ not belonging to $L_{\mathcal{G}_{i}}(a)$ (or $L_{\mathcal{G}_{i-1}}(a)$), which contradicts Lemma \ref{lemmaspliting} (iii).

Since $\mathcal{G}_{i-1}$ is $K_{4,4}^4$-hom-free, then for every $f_1 \in M(x)$ and every $f_2 \in M(a)$, there exist $b_1 \in f_1$ and $b_2 \in f_2$ such that $b_1b_2b_3b_4 \notin \mathcal{G}_{i-1}$ ($b_1b_2b_3b_4 \notin \mathcal{G}_{i}$ as well) for every pair $b_3, b_4 \in V_t$.
Thus for any $w \in W_1\setminus (A_v \cup xaa_1a_2)$, we have
$$d_{\mathcal{B}_i \setminus \mathcal{G}_i}(w)\ge {1\over 3}\cdot|M(x)|\cdot|M(a)|\cdot(|W_2|-2)>\gamma m^3.$$
Thus we get a contradiction with Lemma \ref{lemmaspliting} (iii).

Subcase 2. Adding $A_v$ to $W_2$ minimizes $\Sigma$.

We first prove that all edges which is not good are contained in $W_2\cup A_v$.

Claim 1.2.1: $|A_v|< |Q_{i,1}|- m/20$ (note that $W_1=Q_{i,1}\setminus A_v$). Suppose that $|A_v|\ge |Q_{i,1}|- m/20$.
Then $|W_1|\le m/20$.
Let $w \in A_v$. We bound the number of good edges containing $w$, denoted by $d_{good}(w)$, and the number of edges containing $w$ which are not good, denoted by $d_{bad}(w)$ in $\mathcal{G}_{i-1}$.
Then $$d_{good}(w)\le |W_1|{|W_2|\choose 2} \le {m \over 20} {19m/20 \choose 2}\le 0.023m^3.$$
So $$d_{bad}(w)\ge (9/128-2\alpha) m^3-0.023m^3 \ge  0.047m^3.$$
Now we consider the partition $\{W''_1,W''_2\}$ of $V_t$ obtained by adding $A_v$ to $W_1$ rather than $W_2$.
The number increasing $\Sigma$ is at most
$$|W_1|\cdot {|W_2| \choose 2}\cdot |A_v|+ {|W_1| \choose 2}\cdot |W_2| \cdot |A_v|+ {|W_1| \choose 3}\cdot |A_v|< 0.02 |A_v| m^3,$$
whereas the number decreasing $\Sigma$ is at least
$$ \sum_{w\in A_v}d_{bad}(w) - {|W_1| \choose 2}|W_2|- {|W_1| \choose 3}\ge 0.03|A_v|m^3.$$
This contradicts the minimality of $\Sigma$.

Let $A=W_1$ and $B=W_2$. By Lemma \ref{lemmaspliting}, and the relationship between $\mathcal{G}_{i}$ and $\mathcal{G}_{i-1}$, then $\mathcal{G}_{i-1}$ with $Q_{i-1,1}\cup Q_{i-1,2}$ and $A, B$ satisfy the conditions of Lemma \ref{lemmagen}.
Applying Lemma \ref{lemmagen}, we get that all bad edges are contained in $W_2\cup A_v$. Fix an edge $xyzw \in \mathcal{G}_{i-1}$ satisfying $y, z, w \in W_2$. For $j=1,2$, let
$$B_j(x):=\{x' \in W_j\setminus xyzw: \exists e \in \mathcal{G}_{i-1} \:\: {\rm such} \:\: {\rm that} \:\: \{x,x'\}\subseteq e\}.$$

Claim 1.2.2: $|B_1(x)|\ge 10^{-3}m$ and $|B_2(x)|\ge 10^{-1}m$. We first prove that $|B_1(x)|\ge 10^{-3}m$. Otherwise suppose that $|B_1(x)|< 10^{-3}m$. Now we consider the partition $\{W''_1,W''_2\}$ of $V_t$ obtained by adding $A_v$ to $W_1$ rather than $W_2$.
Since there are no very bad edges and worst edges in the partition of $\{W'_1,W'_2\}$, so every bad edge turns good.
There are at least $|A_v|10^{-3}m^3$ such edges (Recall that every vertex in $A_v$ is bad). On the other side, every good edge containing a vertex of $A_v$ becomes bad.
There are less than $|A_v|\cdot|B_1(x)|m^2 < |A_v|10^{-3}m^3$ such edges. However, this contradicts the minimality of $\Sigma$.
Now we prove that $|B_2(x)|\ge 10^{-1}m$. Since there are at least $10^{-3}m^3$ edges containing $x$ in $W_2\cup A_v$, then ${|B_2(x)| \choose 3} \ge 10^{-3}m^3$.

Claim 1.2.3: Let $M_x$ be the maximum matching of such edges satisfying one vertex in $B_1(x)$ and another three vertices in $B_2(x)$, then $|M_x|\ge 10^{-4}m$.
Otherwise suppose that $|M_x|< 10^{-4}m$. Denote $V_x=\cup_{f \in M_x}f$. Since for every vertex $x_1 \in B_1(x) \setminus V_x$ and every triple $\{x_2,x_3,x_4\} \in B_2(x) \setminus V_x$, $\{x_1,x_2,x_3,x_4\} \in \mathcal{B}_i \setminus \mathcal{G}_{i-1}$ (in $\mathcal{B}_i \setminus \mathcal{G}_i$ as well).
Then there are at least $|B_1(x) \setminus V_x|{|B_2(x) \setminus V_x| \choose 3} \ge \varepsilon n^4$ edges of $\mathcal{B}_i \setminus \mathcal{G}_i$, which contradicts (iii) of Lemma \ref{lemmaspliting}.

Since $\mathcal{G}_{i-1}$ is $K_{4,4}^4$-hom-free, then $\forall f_1 \in M_x$, $\exists a \in f_1$ and $q \in \{y,z,w\}$  such that $aqq'q'' \notin \mathcal{G}_{i-1}$ for all pairs $q',q'' \in V_t$.

Consider the subcase that there are at least $10^{-5}m$ vertices in $B_1(x) \cap (\cup_{f \in M_x})$ such that for every such vertex $b$, $\exists q \in \{y,z,w\}$ such that $bqq'q'' \notin \mathcal{G}_{i-1}$ for all pairs $q',q'' \in V_t$.
By pigeonhole principle, there is $q \in \{y,z,w\}$ such that there are at least $1/3 \cdot 10^{-5}m \cdot {m/2 \choose 2}>10^{-7}m^3$ edges of $\mathcal{B}_i \setminus \mathcal{G}_i$ containing $q$, which contradicts (iii) of Lemma \ref{lemmaspliting}.

In the remaining subcase that there are at least $9\cdot10^{-5}m$ vertices in $B_2(x) \cap (\cup_{f \in M_x})$ such that for every such vertex $b'$, $\exists p \in \{y,z,w\}$ such that $b'pq'q'' \notin \mathcal{G}_{i-1}$ for all pairs $q',q'' \in V_t$.
By pigeonhole principle, there is $p' \in \{y,z,w\}$ such that there are at least ${1 \over 3} \cdot 9 \cdot 10^{-5}m \cdot {m\over 20} \cdot {m\over 2}>10^{-7}m^3$ edges of $\mathcal{B}_i \setminus \mathcal{G}_i$ containing $p'$, which contradicts (iii) of Lemma \ref{lemmaspliting}.

{\em Case 2.} $A_v \cup A_u \subseteq Q_{i,2}$.

Subcase 1. Adding $A_v$ to $W_1$ minimizing $\Sigma$.

Claim 2.1.1: $|A_v|< m/4$. Suppose that $|A_v|\ge m/4$.
In this case, there are no bad edges contained in $W_2$. Consider the vertex $u\in A_u$. So all edges containing $u$ are good. We now count the degree of $u$ in $\mathcal{G}_{i-1}$ (in $\mathcal{G}_{i}$ as well),
$$d_{\mathcal{G}_{i-1}}(u)\le |W_1|{|W_2|\choose 2}\le (m/4+\gamma)\cdot {3m/4-m/4+\gamma \choose 2}< (9/128-2\alpha)m^3,$$
which contradicts Lemma \ref{lemmaspliting} (i).

Let $A=W_1$ and $B=W_2$. By Lemma \ref{lemmaspliting}, and the relationship between $\mathcal{G}_{i}$ and $\mathcal{G}_{i-1}$, then $\mathcal{G}_{i-1}$ with $Q_{i-1,1}\cup Q_{i-1,2}$ and $A, B$ satisfy the conditions of Lemma \ref{lemmagen}. Applying  Lemma \ref{lemmagen}, we are done.
\medskip

Subcase 2. Adding $A_v$ to $W_2$ minimizing $\Sigma$.

Denote
$ W=W_1 \cup W_2=V_t\setminus A_v.$
We first show that there is no edge of $\mathcal{G}_{i-1}$ intersecting $W_1$ with two or three vertices. Let $u' \in A_u$, note that every edge of $\mathcal{G}_{i-1}$ containing $u'$ is good.
Then $(9/128-2\alpha)m^3\le d_{\mathcal{G}_{i-1}}(u')\le |W_1|\cdot {|W_2|\choose 2}$, note that $|W_1|=|Q_{i,1}|=(1/4\pm \gamma)m$, so $|W_2|\ge m/2$.
Let $A=W_1$ and $B=W_2$. By Lemma \ref{lemmaspliting}, and the relationship between $\mathcal{G}_{i}$ and $\mathcal{G}_{i-1}$, then $\mathcal{G}_{i-1}$ with $Q_{i-1,1}\cup Q_{i-1,2}$ and $A, B$ satisfy the conditions of Lemma \ref{lemmagen}. Applying Lemma \ref{lemmagen} we get that there is no edge of $\mathcal{G}_{i-1}$ that contains two or three vertices of $W_1$.
Hence there are at least $10^{-3}m^3$ edges containing $x$ contained in $W_2\cup A_v$.
Fix one such edge $xyzw \in \mathcal{G}_{i-1}$, where $y,z,w \in W_2$.
For $j=1,2$, let
$$B'_j(x):=\{x' \in W_j\setminus \{y,z,w\}: \exists e \in \mathcal{G}_{i-1} \:\: {\rm such} \:\: {\rm that} \:\: \{x,x'\}\subseteq e\}.$$
We claim that $|B'_1(x)|\ge 10^{-4}m$ and $|B'_2(x)|\ge 10^{-2}m$.
We first prove that $|B'_1(x)|\ge 10^{-4}m$. Otherwise suppose that $|B'_1(x)|< 10^{-4}m$. Now we consider the partition $\{W''_1,W''_2\}$ of $V_t$ obtained by adding $A_v$ to $W_1$ rather than $W_2$.
Every bad edge contained in $W_2$ turns good.
There are at least $|A_v|10^{-3}m^3$ such edges. On the other hand, every good edge containing a vertex of $A_v$ becomes bad, every bad edge that intersects $W_1$ with two vertices turns very bad and every very bad edge turns worst. There are less than $|A_v|\cdot|B'_1(x)|m^2 < |A_v|10^{-3}m^3$ such edges. However, this contradicts the minimality of $\Sigma$.
Now we prove that $|B'_2(x)|\ge 10^{-2}m$. Since there are at least $10^{-3}m^3$ edges containing $x$ in $W_2$ and ${|B'_2(x)| \choose 3} \ge 10^{-3}m^3$, then $|B'_2(x)| > 10^{-2}m$.

Let $M'(x)$ be the maximum matching of such edges satisfying one vertex in $B'_1(x)$ and another three vertices in $B'_2(x)$. Similar to Claim 1.2.3, we have  $|M'(x)|\ge 10^{-5}m$.
Since $\mathcal{G}_{i-1}$ is $K_{4,4}^4$-free, then $\forall  f_1 \in M'(x)$, $\exists a \in f_1$ and $ q \in \{y,z,w\}$ such that $aqq'q'' \notin \mathcal{G}_{i-1}$ for all pairs $q',q'' \in V_t$.

Consider the subcase that there are at least $10^{-6}m$ vertices in $D \subseteq B'_1(x) \cap (\cup_{f \in M(x)})$ such that $\forall b \in D $, $\exists q \in \{y,z,w\}$ such that $bqq'q'' \notin \mathcal{G}_{i-1}$ for all pairs $q',q'' \in V_t$.
By pigeonhole principle, there is $q \in \{y,z,w\}$ such that there are at least $1/3 \cdot 10^{-6}m \cdot {m/2 \choose 2}>10^{-8}m^3$ edges of $\mathcal{B}_i \setminus \mathcal{G}_i$ containing $q$, which contradicts (iii) of Lemma \ref{lemmaspliting}.

In the remaining subcase that there are at least $9\cdot10^{-6}m$ edges $e \in M'(x)$ such that
 $\exists d \in e\cap W_2$ and $ d' \in \{y,z,w\}$ such that $dd'q'q'' \notin \mathcal{G}_{i-1}$ for all pairs $q',q'' \in V_t$.
By pigeonhole principle, there is $p' \in \{y,z,w\}$ such that there are at least ${1 \over 3} \cdot 9 \cdot 10^{-6}m \cdot {m\over 20} \cdot {m\over 2}>10^{-8}m^3$ edges of $\mathcal{B}_i \setminus \mathcal{G}_i$ containing $p'$, which contradicts (iii) of Lemma \ref{lemmaspliting}.
\qed
\medskip

In our next lemma we will conclude the second stage of the proof by showing that every edge of $\mathcal{G}_{i-1}$ is good. First we show that $|W'_1|\ge m/5$ and $|W'_2|\ge 3m/5$ (so $m/5 \le |W'_1|\le 2m/5$ and $3m/5 \le |W'_2|\le 4m/5$). Note that $W'_1$ and $W'_2$ are obtained by adding $A_v$ to $ W_1$ or $W_2$.
If $W'_1= Q_{i,1}$ then this holds by Lemma \ref{lemmaspliting} (ii). Now we consider the remaining two cases. In both cases, let $x \in A_v$.

Case 1. Adding $A_v$ from $ Q_{i,1}$ to $ W_2$. For this case $ Q_{i,2} \subseteq W'_2$, so it suffices to prove $|W'_1|\ge m/5$. Suppose that $|W'_1|< m/5$. By Lemma \ref{lemmaspliting} (ii) and Lemma \ref{lemmabad}, we have
\begin{eqnarray*}
d_{\mathcal{G}_{i-1}}(x)
&\le& 10^{-3}m^3+|W'_1| {|W'_2| \choose 2} \\
&\le& 10^{-3}m^3+{m \over 5} \cdot {8m^2 \over 25} \\
&<& ({9 \over 128}-\alpha)m^3,
\end{eqnarray*}
\end{proof}
which contradicts Lemma \ref{lemmaspliting} (i).

Case 2. Adding $A_v$ from $ Q_{i,2}$ to $ W_1$. For this case $ Q_{i,1} \subseteq W'_1$, so it suffices to prove $|W'_2| \ge 3m/5$. Otherwise suppose $|W'_2|< 3m/5$. By Lemma \ref{lemmaspliting} (ii) and Lemma \ref{lemmabad}, we have
\begin{eqnarray*}
d_{\mathcal{G}_{i-1}}(x)
&\le& 10^{-3}m^3+{|W'_2| \choose 3} \\
&<& ({9 \over 128}-\alpha)m^3,
\end{eqnarray*}
which contradicts Lemma \ref{lemmaspliting} (i). We can now state our next lemma.

\begin{lemma}\label{lemmabad1}
Every edge of $ \mathcal{G}_{i-1} $ is good.
\end{lemma}
\begin{proof}
Assume for the sake of contradiction that $e \in \mathcal{G}_{i-1}$ is not a good edge.

{\em Case 1.} $|e \cap W'_1|\ge 2$.

Let $a, b \in e$ satisfy $a, b \in W'_1$.
For $x=a$ or $b$ let
$$B(x):=\{x' \in W'_2 \setminus e: \exists e' \in \mathcal{G}_{i-1} \:\: {\rm such} \:\: {\rm that} \:\: \{x,x'\}\subseteq e'\}.$$
We claim that $ |B(a)\cap B(b)|\ge m/2$.
Otherwise $ |B(a)|\le |W'_2|/2+ m/4$ or $ |B(b)|\le |W'_2|/2+ m/4$. Without loss of generality we can assume that $ |B(a)|\le |W'_2|/2+ m/4$, then $ |B(a)|\le 2m/5+ m/4=13m/20$. Hence $d_{\mathcal{G}_{i-1}}(a)\le {13m/20 \choose 3}+10^{-3}m^3< ({9 \over 128}-\alpha)m^3$, which contradicts Lemma \ref{lemmaspliting} (i).

We claim that there are at least $ m/100$ disjoint triples in $B(a)\cap B(b)$ belonging to $L_{\mathcal{G}_{i-1}}(a)$. Otherwise
$d_{\mathcal{G}_{i-1}}(a)\le {4m/5 \choose 3}-{m/2-3m/100 \choose 3}+10^{-3}m^3< ({9 \over 128}-\alpha)m^3$, which contradicts Lemma \ref{lemmaspliting} (i). Fix such a set of $ m/100$ triples as $M_a$.

We also claim that there are at least $ m/100$ disjoint triples in $(B(a)\cap B(b))\setminus (\cup_{f \in M_a}f)$ belonging to $L_{\mathcal{G}_{i-1}}(b)$. Otherwise
$d_{\mathcal{G}_{i-1}}(b)\le {4m/5 \choose 3}-{m/2-6m/100 \choose 3}+10^{-3}m^3< ({9 \over 128}-\alpha)m^3$, which contradicts Lemma \ref{lemmaspliting} (i). Fix such  a set of $ m/100$ triples as $M_b$.

Since $\mathcal{G}_{i-1}$ is $K_{4,4}^4$-hom-free, then $\forall  f_1 \in M_a$ and $ f_2 \in M_b$, $\exists w_1\in f_1$ and $w_2\in f_2$ such that $\{w_1,w_2\}$ is not contained in any edge of $\mathcal{G}_{i-1}$.
Let $w_3 \in W_2 \setminus (\cup_{f \in M_a\cup M_b}f)$ (there exists such a vertex), then $d_{\mathcal{B}_i \setminus \mathcal{G}_i}(w_3)\ge |M(a)|\cdot|M(b)|\cdot(|W_1|-m/50)>\gamma m^3$
, which contradicts Lemma \ref{lemmaspliting} (iii).

{\em Case 2.} $e=xyzw \subseteq W'_2$.

For $j=1,2$ and $a\in \{x,y,z,w\}$ let
$$B_j(a):=\{a' \in W'_j: \exists e' \in \mathcal{G}_{i-1} \:\: {\rm such} \:\: {\rm that} \:\: \{a,a'\}\subseteq e'\}.$$
We claim that $|B_1(a)|\ge m/6$ and $|B_2(a)|\ge m/2$ for every $a\in \{x,y,z,w\}$.
We first prove that $|B_1(a)|\ge m/6$. Otherwise suppose that $|B_1(a)|< m/6$.
Then $d_{\mathcal{G}_{i-1}}(a)< 10^{-3}m^3+|B_1(a)|{|W'_2| \choose 2} \le 10^{-3}m^3+{m \over 6}{4m/5 \choose 2} <({9 \over 128}-\alpha)m^3$,
which contradicts Lemma \ref{lemmaspliting} (i).
Now we prove that $|B_2(a)|\ge m/2$. Otherwise, in a similar way as above, $d_{\mathcal{G}_{i-1}}(a)<10^{-3}m^3+|W'_1|{m/2 \choose 2}\le 10^{-3}m^3+2m/5{m/2 \choose 2} <({9 \over 128}-\alpha)m^3$, which contradicts Lemma \ref{lemmaspliting} (i).

Now we show that almost all vertices in $ B_1(a)\cup B_2(a)$ are adjacent to $y,z,w$. Otherwise without loss of generality assume there are $10\gamma m$ vertices of $B_1(a)$ which are not adjacent to $y$. Then
$$d_{\mathcal{B}_i\setminus \mathcal{G}_i}(y)> 10\gamma m {3m/5 \choose 2}>\gamma m^3,$$
 contradicting to Lemma \ref{lemmaspliting} (iii).
Let $B_1 \subseteq W'_1$ and $B_2 \subseteq W'_2$ be the common neighbourhood of $x,y,z,w$.
 Since almost all vertices in $V_t$ are the neighbours of every vertex in $W'_2$, so we can assume that $|B_1|\ge m/7$ and $|B_2|\ge m/3$.
 It is clearly that there is an edge $e' \in \mathcal{G}_{i-1}$ with one vertex in $B_1$ and another three vertices in $B_2$.
 Then $\forall a \in xyzw$ and $ b \in e'$, there exists one edge $e_{ab}$ of $\mathcal{G}_{i-1}$ such that $\{a,b\}\subseteq e_{ab}$. Hence $\{e_{ab}:a \in \{x,y,z,w\}, b \in e'\}\cup \{xyzw,e'\}$ forms a configuration contradicting that $\mathcal{G}_{i-1}$ is $K_{4,4}^4$-hom-free.
\qed
\end{proof}
\medskip

This shows that $\mathcal{Q}_{i-1}$ satisfies (P3), so splitting has the required properties. It terminates with $\mathcal{Q}_{0}$ such that $\mathcal{F}[V_t]=\mathcal{H}_{0}[V_t] \subseteq Q_{0,1}\times {Q_{0,2} \choose 3}$.
Let $A=Q_{0,1}\cup (V(\mathcal{F})\setminus V_t)$ and $B=Q_{0,2}$. Then by Lemma \ref{lemmav}, we have $|\{e \in \mathcal{F}: |e \cap A| \ge 2\}|+|\{e \in \mathcal{F}: e \subseteq B\}| < \varepsilon n^4 $. This concludes the proof of Theorem \ref{theostability}.

\subsection{Proof of Theorem \ref{theorem}}

Let $\mathcal{F}=(V,E)$ be a maximum  $K_{4,4}^4$-free $4$-graph on $n$ vertices, where $n$ is sufficiently large. Since $S^4(n)$ is $K_{4,4}^4$-free, $|\mathcal{F}|\ge |S^4(n)|$.
Similar to \cite{HK} (in the first paragraph of Section 4.2), it suffices to prove Theorem \ref{theorem} under the  assumption that the minimum degree of $\mathcal{F}$ is at least $\delta(S^4(n))$.
Indeed, assume we have proved Theorem \ref{theorem} for every maximum $K_{4,4}^4$-free $4$-graph $\mathcal{F}$ on $n\ge n_0$ vertices and minimum degree at least $\delta(S^4(n))$.
Let $\mathcal{H}_n$ be a maximum $K_{4,4}^4$-free $4$-graph on $n\ge n_0^4$ vertices, delete the vertex of $\mathcal{H}_n$ with degree less than $\delta(S^4(n))$ until $\delta(\mathcal{H}_m)\ge \delta(S^4(m))$ or all vertices have been deleted.
By some easy calculations we can show that $m\ge n_0$. But $\mathcal{H}_m$ has minimum degree at least $\delta(S^4(m))$ and strictly more than $e(S^4(m))$ edges, which contradicts our assumption.

Let $c_1, c_2, c_3$ and $\varepsilon>0$ be real numbers satisfying $$\varepsilon \ll c_3 \ll c_2 \ll c_1 \ll 1.$$
Let $V=W_1\cup W_2$ be a partition of the vertex set of $\mathcal{F}$ which minimizes
$$\Sigma':=|\{e\in E: |e \cap W_1|=2 \:\: {\rm or } \:\: e\subseteq W_2\}|+2|\{e\in E: |e \cap W_1|=3\}|+3|\{e\in E: e \subseteq W_1\}|.$$
By Theorem \ref{theostability} we can assume that $\Sigma' <\varepsilon n^4$. Similar to before, we have
$$|W_1|=({1 \over 4}\pm c_3)n \:\: {\rm and} \:\: |W_2|=({3 \over 4}\pm c_3)n.$$

Similar to the proof of Theorem \ref{theostability}, we call an edge $e \in E$ {\em bad} if $|e \cap W_1|=2$ or $e \subseteq W_2$, {\em very bad}  if $|e \cap W_1| \ge 3$, {\em good} otherwise. Equivalently, we need to show that $\Sigma'=0$.
We say that a vertex $v\in V$ is {\em bad} if it is contained in at least $30c_1 n^3$ edges which are not good.
Before proving that all edges are good, we will prove that any vertex of $V$ cannot be contained in too many edges which are not good.

\begin{lemma}\label{lemmabad2}
There are no bad vertices.
\end{lemma}
\begin{proof}
Assume for the sake of contradiction that $x \in V$ is a bad vertex. First suppose that $x \in W_1$. Then all edges containing $x$ which are not good intersect $W_1$ with at least two vertices.
Denote
$F_x=\{e\in E: x \in e, |e \cap W_1|\ge 2\}$ and $I_x=\{f \in {V\choose 3}: f \cup \{x\} \in F_x\}$. Then
$|F_x|=|I_x|\ge 30c_1 n^3$.
\newline
Claim 1. There is a set $K \subseteq I_x$ with $|K|=10c_1n$ such that $\forall f_1, f_2 \in K$, $f_1 \cap f_2=\emptyset$. This is because that $\forall f \in I_x$, $f$ intersects less than $3{n \choose 2}$ elements of $I_x$.
Denote
$$B(x):=\{v \in W_2\backslash (\cup_{f\in I_x}f):\: \{v,x\} \:{\rm is} \: {\rm contained} \: {\rm in} \:{\rm at} \:{\rm least} \: 40n \: {\rm edges} \:  {\rm of} \: \mathcal{F}\}.$$
Claim 2. $|B(x)|\ge n/20$. By the minimality of $\Sigma'$, the number of good edges of $\mathcal{F}$ containing $x$ is no less than the number of edges in $\mathcal{F}$ which are not good containing $x$. Then
$|W_2 \setminus B(x)|40n+|B(x)|{|W_2|\choose 2}>d(x)/2\ge (9/256-c_1/2)n^3$, so $|B(x)|\ge n/20$.
\newline
Claim 3. There are at least $|K|/2 = 5c_1n$ elements $f=\{a,b,c\}\in K$ with $a \in W_1$ such that there is $B_{f_a} \subseteq B(x)$ of size $c_1n$ satisfying that $\{v, a\}$ is contained in at least $40n$ edges of $\mathcal{F}$ for every $v \in B_{f_a}$.
Otherwise there are at least $|K|/2$ elements $f=\{a,b,c\}\in K$ with $a \in W_1$ such that for all but at most $c_1n$ vertices $v$ in $B(x)$, $\{v, a\}$ is contained in at most $40n$ edges of $\mathcal{F}$.
Then there are at least ${|K| \over 2} \left({|B(x)\setminus B_{f_a}| \choose 3}-40n^2\right)> \varepsilon n^4$ edges in $\left(W_1\times {W_2 \choose 3}\right)\setminus \mathcal{F}$, contradicting that $|E|\ge |S^4(n)|$.
Denote the set of elements $f \in K$ satisfying the property above as  $K'$.
\newline
Claim 4. There is $f \in K'$ with $a \in W_1 \cap f$ such that there is a set $M_{f_a} \subseteq {B_{f_a} \choose 3}$ of size $c_2n$ satisfying $e_1 \cup \{a\}\in \mathcal{F}$, $e_1\cup \{x\} \in \mathcal{F}$ and $e_1 \cap e_2=\emptyset$ for all different $e_1, e_2 \in M_{f_a}$.
Otherwise for each $f \in K'$ with $a \in W_1 \cap f$, fix a maximum set $M_{f_a} \subseteq {B_{f_a} \choose 3}$ satisfying $e_1 \cup \{a\}$, $e_1\cup \{x\} \in \mathcal{F}$ and $e_1 \cap e_2=\emptyset$ for all different $e_1, e_2 \in M_{f_a}$.
Then $|M_{f_a}|<c_2n$ and for every triple $g \in {B_{f_a} \setminus (\cup_{e'\in M_{f_a}}e') \choose 3}$, $g\cup \{a\} \notin \mathcal{F}$ or $g\cup \{x\} \notin \mathcal{F}$ (or both).
So there are at least ${|K| \over 2}{|B_{f_a}|-3c_2n \choose 3}> \varepsilon n^4$ edges in $\left(W_1\times {W_2 \choose 3}\right)\setminus \mathcal{F}$, contradicting that $|E|\ge |S^4(n)|$.

Since $\mathcal{F}$ is $K_{4,4}^4$-free, for every pair $e_1,e_2 \in M_{f_a}$, $\exists u \in e_1$ and $v \in e_2$ such that there are at most $40n$ edges of $\mathcal{F}$ containing $\{u,v\}$.  Otherwise we can greedily choose vertices to extend $e_1\cup \{x\}$ and $e_2\cup \{a\}$  to a copy of $K_{4,4}^4$.
Hence there are at least ${|M_{f_a}|\choose 2}\left(|W_1||W_2\backslash (\cup_{e\in M_{f_a}}e)|-40n\right)> \varepsilon n^4$ edges in $\left(W_1\times {W_2 \choose 3}\right)\setminus \mathcal{F}$, contradicting that $|E|\ge |S^4(n)|$.

Now suppose that $x \in W_2$. We divide it into two cases according to whether there are half  edges containing $x$ which are not good contained in $W_2$. The proof of both cases is similar to the above.
We first prove the case that there are at least $15c_1 n^3$ edges containing $x$ which are not good contained in $W_2$.
Fix a set of those edges removing $x$ as $L'(x)$. Consider a maximum matching $M(x)$ in $L'(x)$.
Then $|M(x)|\ge c_1 n$. Fix $M'\subseteq M(x)$ with $|M'|=c_2 n$. Note that $\bigcup_{f\in M'}f \subseteq W_2$.
For $j=1,2$, let
$$B'_j(x):=\{x' \in W_j: \: {\rm there} \:{\rm are} \:{\rm at} \:{\rm least} \: 40n \: {\rm edges} \:  {\rm of} \: \mathcal{F} \: {\rm containing} \:\{x,x'\}\}.$$
Claim 5. $|B'_1(x)|\ge 3c_2 n$ and $|B'_2(x)|\ge 5c_2 n$. Since there are at least $15c_1 n^3$ edges  containing $x$ which are not good such that each is contained in $W_2$, it is clearly that $|B'_2(x)|\ge 5c_2 n$.
If $|B'_1(x)|< 3c_2 n$, we get a new partition by moving $x$ from $W_2$ to $W_1$ with smaller $\sum$, which contradicts the minimality of $\Sigma'$.
Fix $B''_1(x) \subseteq B'_1(x)$ and $B''_2(x) \subseteq B'_2(x)\setminus (\bigcup_{f\in M'}f)$ with $|B''_1(x)|=2c_2 n$ and $|B''_2(x)|=4c_2 n$.
Consider a maximum matching of $\mathcal{F}$, denoted as $M''$, with one vertex in $B''_1(x)$ and another three vertices in $B''_2(x)$.
\newline
Claim 6. $|M''|\ge c_2 n$. Otherwise there are at least $c_2 n {c_2 n \choose 3}>\varepsilon n^4$ edges in $(W_1\times {W_2 \choose 3})\setminus \mathcal{F}$, a contradiction.
Since $\mathcal{F}$ is $K_{4,4}^4$-free, for every $f \in M'$ and $e \in M''$ there are $a \in f$ and $b \in e$ such that there are at most $40n$ edges of $\mathcal{F}$ containing $\{a,b\}$; otherwise we can greedily choose vertices to extend $f\cup \{x\}$ and $e$ to a copy of $K_{4,4}^4$.
Hence we can find more than $|M'|\cdot|M''|\cdot{n\over 5}\cdot{3n\over 5}-40n^3>\varepsilon n^4$ elements in $\left({W_1\times {W_2 \choose 3}}\right) \setminus \mathcal{F}$ and this contradicts that $|E|\ge |S^4(n)|$.

The last case is that  there are at least $15c_1 n^3$ edges containing $x$ which are not good such that each intersects $W_1$ with at least two vertices.
Denote
$F'_x=\{e\in \mathcal{F}: x \in e, |e \cap W_1|\ge 2\}$ and $I'_x=\{f: f \cup \{x\} \in F_x\}$. Since $\forall f \in I'_x$, $f$ intersects less than $3{n \choose 2}$ elements of $I'_x$, there is a set $P\subseteq I'_x$ with size $5c_1n$ such that $\forall f_1, f_2 \in P$, $f_1 \cap f_2=\emptyset$.

Claim 7. There are at least $c_1 n$ elements $f=\{y,z,w\} \in P$, denoted as $P'$, where $y,z \in W_1$ such that there is a set $V'_{f_{yz}} \subseteq W_2$ with size at least $n/10$ such that $\{y,v\}$ and $\{z,v\}$ are both contained in at least $40n$ edges of $\mathcal{F}$ for every $v \in V'_{f_{yz}}$;
otherwise there are at least $c_1 n\cdot \left({|W_2|-n/10 \choose 3}-40n^2\right)>\varepsilon n^4$ edges in $(W_1\times {W_2 \choose 3})\setminus \mathcal{F}$, a contradiction.

For every $f=\{y,z,w\} \in P'$, where $y,z \in W_1$ consider the set $J'_{yz} \subseteq {V'_{f_{yz}}\setminus f \choose 3}\cap L(y)\cap L(z)$ satisfing all elements of $J'_{yz}$ are pairwise disjoint.

Claim 8. There is at least an element $f'=\{y',z',w'\}$, where $y',z' \in W_1$ such that $|J'_{y'z'}|\ge c_1 n$; otherwise there are at least $|P'|\cdot {|V'_{f'_{y'z'}}|-3c_1 n \choose 3}>\varepsilon n^4$ edges in $\left(W_1\times {W_2 \choose 3}\right)\setminus \mathcal{F}$, a contradiction.

Since $\mathcal{F}$ is $K_{4,4}^4$-free, for every pair $f_1, f_2 \in J'_{y'z'}$, there are $a \in f_1$ and $b \in f_2$ such that there are at most $40n$ edges of $\mathcal{F}$ containing $\{a,b\}$; otherwise we can greedily choose vertices to extend $f_1 \cup \{y'\}$ and $f_2\cup \{z'\}$ to a copy of $K_{4,4}^4$.
Hence we can find at least $ {|J'_{y'z'}|\choose 2}\left(|W_1|\cdot |W_2\setminus (\cup_{f \in J'_{y'z'}}f)|-40n\right)>\varepsilon n^4$ edges in $(W_1\times {W_2 \choose 3})\setminus \mathcal{F}$, a contradiction.
\end{proof}
\qed
\medskip

Finally, we show that every edge of $\mathcal{F}$ is good. Suppose that $e=xyzw \in \mathcal{F}$ is an edge which is not good.
First we assume that $|e \cap W_1|\ge 2$.
Without loss of generality assume that $x, y \in W_1$.
Denote
$$B(x):=\{v \in W_2\backslash e:\: \{v,x\} \:{\rm is} \: {\rm contained} \: {\rm in} \:{\rm at} \:{\rm least} \: 40n \: {\rm edges} \:  {\rm of} \: \mathcal{F}\}.$$
We claim that $|B(x)|\ge 7n/10$. Otherwise
$d_{\mathcal{F}}(x)\le {7n/10 \choose 3}+40 n^2+ 30c_1 n^3<\delta(S^4(n))$, a contradiction.
Consider the family of triples $L_x \subseteq {B(x) \choose 3}$ satisfying $f \cap g =\emptyset$ for every pair $f, g \in L_x$.
Similar as before, we claim that $|L_x|\ge n/5$. Otherwise suppose that $|L_x|< n/5$, then $d_{\mathcal{F}}(x)\le {|W_2| \choose 3}-\left({|W_2| -3n/5 \choose 3}-40n^2\right)+ 30c_1 n^3<\delta(S^4(n))$, a contradiction. Fix $L'_x\subseteq L_x$ with $|L'_x|=n/10 $.
Let $$B(y):=\{v \in B(x)\setminus (\cup_{e\in L'_x}e):\: \{v,y\} \:{\rm is} \: {\rm contained} \: {\rm in} \:{\rm at} \:{\rm least} \: 40n \: {\rm edges} \:  {\rm of} \: \mathcal{F}\}$$
and
$L_y \subseteq {B(y) \choose 3}$ satisfying $f \cap g =\emptyset$ for every pair $f, g \in L_y$.
Similarly, we have $|L_y|\ge n/10$.

Since $\mathcal{F}$ is $K_{4,4}^4$-free, then for every $f_1 \in L'_x$ and $f_2 \in L_y$, there are $a \in f_1$ and $b \in f_2$ such that there are at most $40n$ edges of $\mathcal{F}$ containing $\{a,b\}$;
otherwise we can greedily choose vertices to extend $f_1\cup \{x\}$ and $f_2\cup \{y\}$ to form a copy of $K_{4,4}^4$. Hence we can find at least $|L'_x|\cdot|L_y|\cdot {n \over 5} \cdot {3n \over 5}-40 n^3  >\varepsilon n^4$ edges in $\left(W_1\times {W_2 \choose 3}\right)\setminus \mathcal{F}$, a contradiction.
\medskip

Now assume that $e=xyzw \subseteq W_2$.
For $j=1,2$, let
$$B_j(e):=\{v \in W_j: \: {\rm there} \:{\rm are} \:{\rm at} \:{\rm least} \: 40n \: {\rm edges} \:  {\rm of} \: \mathcal{F} \: {\rm containing} \:\{u,v\} \: {\rm for} \:  {\rm every} \: u \in e\}.$$
We claim that $|B_1(e)|\ge n/5$ and $|B_2(e)|\ge 3n/5$.
Otherwise, first suppose that $|B_1(e)|< n/5$. Since for
every $v\in W_1 \setminus B_1(e)$, there is at least one $u\in e$ such that $\{u,v\}$ is contained in less than $40n$ edges of $\mathcal{F}$.
So there exists at least one $u\in e$ and at least ${1 \over 4}|W_1 \setminus B_1(e)|$ vertices $v\in W_1 \setminus B_1(e)$ such that $\{u,v\}$ is contained in less than $40n$ edges of $\mathcal{F}$.
Then $ d_{\mathcal{F}}(u) \le {1 \over 4} \cdot 40n|W_1\setminus B_1(e)|+\left(|B_1(e)|+{3 \over 4}|W_1\setminus B_1(e)|\right){|W_2| \choose 2}+ 30c_1 n^3 \le 10n^2+\left({3\over 4}({1\over 4}+c_3 )n+{n \over 20}\right){(3/4+c_3 )n \choose 2}+ 30c_1 n^3 < \delta(S^4(n)) $, a contradiction.
Now suppose that $|B_2(e)|< 3n/5$. Similarly, there exists $u\in e$ such that
$ d_{\mathcal{F}}(u) \le {1 \over 4} \cdot 40n|W_2\setminus B_2(e)|+|W_1|{|B_2(u)|+{3 \over 4}|W_2\setminus B_2(e)| \choose 2}+ 30c_1 n^3 <\delta(S^4(n)) $, a contradiction.

Let $M$ be a maximum matching of edges with one vertex in $B_1(e)$ and another three vertices in $B_2(e)$. We claim that $|M|\ge n/6$.
Otherwise there is a vertex $u \in B_1(e)$ such that there are at least ${{3n\over 5}-3{n\over 6} \choose 3}$ triples in $B_2(e)$ not belonging to $L_{\mathcal{F}}(u)$. Then $d_{\mathcal{F}}(u)\le {({3 \over 4}+ c_3)n \choose 3}-{n/10 \choose 3}+30c_1 n^3<\delta(S^4(n)) $, a contradiction.

In fact, as long as there is one edge (good edge) $e'$ in $M$ we would get a contradiction, since we can find a copy of $K_{4,4}^4$ as follows:
for each vertex $u \in e'$ and  each vertex $v \in xyzw$, we can find an edge $e_{uv}$ of $\mathcal{F}$ greedily  such that the other two vertices of all these $e_{uv}$ are all different. This contradicts that $\mathcal{F}$ is $K_{4,4}^4$-free.
Then
$\mathcal{F}$ is isomorphic to $W_1\times{W_2 \choose 3}$. By the maximality of the number of edges of $\mathcal{F}$, $|W_1|=\lfloor{1 \over 4}n\rfloor$, or $|W_1|=\lceil{1 \over 4}n\rceil$ for $n \equiv 3$ {\rm (mod $4$)}.
\qed
\medskip

{\bf Remarks.} We learned that
 Norin, Watts, and Yepremyan\cite{NWY} had
 a proof of Conjecture \ref{HFcon} for all $r$. Our work here is independent of theirs. This manuscript uses some ideas from  \cite{JPW} such as reducing to  left compressed and dense hypergraphs and  the method may apply to $r$-graphs
without a matching of any fixed size
as in \cite{JPW} for $r=3$. The method can also be extended
to prove Conjecture \ref{HFcon} for $r=5$.

\end{document}